\documentclass[oneside,a4paper]{amsart} 
\usepackage{amssymb,verbatim,mathtools}

\usepackage[shortlabels]{enumitem}
    \setenumerate[0]{label={\rm (\roman*)}, leftmargin=*}

\usepackage[T1]{fontenc}
\usepackage[english]{babel}

\usepackage[textsize=footnotesize,textwidth=20ex,colorinlistoftodos]{todonotes}

\usepackage{hyperref}
\hypersetup{
    colorlinks=true,
    linkcolor=blue,
    filecolor=magenta, 
    citecolor=red,     
    urlcolor=blue,
}

\usepackage[capitalize]{cleveref}
    \crefname{enumi}{}{}
    \Crefname{enumi}{Item}{Items}
    \crefname{equation}{}{}
    \Crefname{equation}{Equation}{Equations}

\newtheorem{theorem}{Theorem}[section]
\newtheorem*{maintheorem}{Main Theorem}
\newtheorem{lemma}[theorem]{Lemma}
\newtheorem{proposition}[theorem]{Proposition}
\newtheorem{corollary}[theorem]{Corollary}

\theoremstyle{definition}
\newtheorem{remark}[theorem]{Remark}
\newtheorem{definition}[theorem]{Definition}
\newtheorem{notation}[theorem]{Notation}
\newtheorem{construction}[theorem]{Construction}
\newtheorem{assumption}[theorem]{Assumption}
\newtheorem{example}[theorem]{Example}

\DeclareMathOperator{\Inst}{Instrl}
\DeclareMathOperator{\Inder}{Inder}
\DeclareMathOperator{\Der}{Der}
\DeclareMathOperator{\Aut}{Aut}
\DeclareMathOperator{\id}{id}
\DeclareMathOperator{\ad}{ad}
\DeclareMathOperator{\Char}{char}

\DeclareMathOperator{\End}{End}
\DeclareMathOperator{\Sym}{Sym}

\DeclareMathOperator{\sign}{sgn}

\newcommand{\A}{\mathcal{A}}
\renewcommand{\SS}{\mathcal{S}}
\newcommand{\HH}{\mathcal{H}}
\newcommand{\LL}{\mathcal{L}}

\newcommand{\E}{\mathcal{E}}
\newcommand{\F}{\mathcal{F}}
\newcommand{\M}{\mathcal{M}}
\newcommand{\dash}{\nobreakdash-\hspace{0pt}}

\numberwithin{equation}{section}
\makeatletter
\@namedef{subjclassname@2020}{\textup{2020} Mathematics Subject Classification}
\makeatother

\begin{document}

\title[Inner ideals and structurable algebras]{Inner ideals and structurable algebras: \\ Moufang sets, triangles and hexagons}

\author{Tom De Medts}
\address[T. De Medts and J. Meulewaeter]{Department of Mathematics: Algebra and Geometry, Ghent University, Krijgslaan 281--S25, 9000 Gent\\Belgium}
\email[Tom De Medts]{tom.demedts@ugent.be}

\author{Jeroen Meulewaeter}
\email[Jeroen Meulewaeter]{jeroen.meulewaeter@ugent.be}
\thanks{The second author is a PhD Fellow of the Research Foundation Flanders (Belgium) (F.W.O.-Vlaanderen), 166032/1128720N}

\date{\today}

\begin{abstract}
    We construct Moufang sets, Moufang triangles and Moufang hexagons using inner ideals of Lie algebras obtained from structurable algebras via the Tits--Kantor--Koecher construction. The three different types of structurable algebras we use are, respectively:
    (1) structurable division algebras,
    (2) algebras $D \oplus D$ for some alternative division algebra $D$, equipped with the exchange involution,
    (3) matrix structurable algebras $M(J,1)$ for some cubic Jordan division algebra $J$.
    In each case, we also determine the root groups directly in terms of the structurable algebra.
\end{abstract}

\maketitle

\section{Introduction}

\emph{Spherical buildings} have been introduced by Jacques Tits \cite{Tits1974} as a tool to study isotropic simple linear algebraic groups over arbitrary fields. These spherical buildings always satisfy the so-called \emph{Moufang property}, which says that the automorphism groups of such a building is highly transitive (in a very precise way). If the rank of the building (which coincides with the relative rank of the algebraic group) is $1$, then the building is called a \emph{Moufang set}; if it is $2$, then the building is called a \emph{Moufang polygon}.
Depending on the relative type, the Moufang polygon will be a \emph{Moufang triangle} (relative type $A_2$), \emph{Moufang quadrangle} (relative type $B_2$ or $BC_2$) or a \emph{Moufang hexagon} (relative type $G_2$).
The Moufang polygons have been classified and investigated in detail in \cite{Tits2002}.

Each linear algebraic group has an associated Lie algebra. If the algebraic group is isotropic and the underlying field has characteristic not $2$ or $3$, then the Lie algebra (or, more precisely, its derived algebra) also arises via the Tits--Kantor--Koecher (TKK) construction starting from a \emph{structurable algebra}, often in more than one way \cite[Theorem 5.9]{Stavrova2017}.

In this paper, we discuss constructions of Moufang sets, triangles and hexagons directly from some specific structurable algebras. 
We construct these geometries using \emph{inner ideals} of Lie algebras.
A subspace $I$ of a Lie algebra $\LL$ is called an inner ideal if $[I, [I, \LL]] \leq I$; see \cref{def inner ideal} below.
Inner ideals in Lie algebras have been introduced and investigated by Georgia Benkart in her PhD thesis \cite{Benkart1974}; see also \cite{Benkart1977}.

We can summarize our results as follows. 
\begin{maintheorem}
Let $\A$ be a structurable algebra and let $\mathcal{G}$ be the poset of all proper non-trivial inner ideals of the TKK Lie algebra associated to $\A$.
\begin{enumerate}
    \item\label{main:1} If $\A$ is a structurable division algebra, then $\mathcal{G}$ forms a Moufang set (\cref{non-ab Moufang,Set Inner are Moufang}).
    \item\label{main:2} If $\A = D \oplus D$ for some alternative division algebra $D$, equipped with the exchange involution, then $\mathcal{G}$ forms the dual double of a Moufang triangle (\cref{triangle root groups}). 
    \item\label{main:3} If $\A = M(J, 1)$ for some cubic Jordan division algebra $J$, then $\mathcal{G}$ forms a Moufang hexagon (\cref{Hexagon Main Theorem,hexagon root groups}).
\end{enumerate}
\end{maintheorem}

In each case, we also describe the root groups of these geometries in terms of the associated structurable algebras.

Notice that in part \cref{main:2} of our main theorem, we obtain the Moufang triangles only via their dual double (which is a thin generalized hexagon). Indeed, it is impossible to construct Moufang triangles directly as the geometry $\mathcal{G}$ for some structurable algebra.
See \cref{rem:dualdouble} below.

\medskip

In \cite{Cohen2006}, Arjeh Cohen and Gabor Ivanyos have introduced \emph{extremal geometries} associated to Lie algebras.
An element of a Lie algebra is called \textit{extremal} if it spans a one-dimensional inner ideal,
and the corresponding extremal geometry has as point set the set of all those one-dimensional inner ideals.
It seems that the first appearance of extremal geometries is \cite[Chapter 12]{Faulkner1977}, where John Faulkner discusses generalized hexagons.
There are some connections between this result and our results in \cref{sec 7}; see \cref{Remark Faulkner hexagon} below.

However, if there are no one-dimensional inner ideals, then the extremal geometry is empty.
In these cases, one can take the \textit{minimal proper inner ideals} as points.
Only some of our constructions are examples of extremal geometries:
in \cref{sec 3,sec 4,sec 6} we describe geometries which are not extremal geometries.

If the Lie algebra is defined over an algebraically closed field of characteristic~$0$, then its inner ideals have been studied in detail in \cite{Draper2012}.
Moreover, under the same assumptions, Faulkner has connected these inner ideals to geometries; see \cite{Faulkner1973}.

The concept of an inner ideal also exists in Jordan theory. 
(In fact, it was introduced in Jordan algebras before it was introduced in Lie algebras.)
The inner ideals of Jordan algebras have been studied (and in many cases classified) in \cite{McCrimmon1971b}.
The correspondence between inner ideals of a Jordan algebra and inner ideals of its TKK Lie algebra is obvious.

In \cite{Garibaldi2001}, Skip Garibaldi defines inner ideals in structurable algebras of skew-dimension one%
\footnote{Interestingly, Garibaldi shows in \cite[\S 7]{Garibaldi2001} that some of the inner ideals of the Brown algebra are related to a building of type $E_7$, which is a result in the spirit of the current paper.},
and this is also what we will use in \cref{sec 7}.
To the best of our knowledge, inner ideals in structurable algebras in general have not yet been considered in the literature, but we can do without such a notion in \cref{sec 4,sec 5,sec 6}.
Our rather technical proof of \cref{Characterization extremal J +} below suggests that it is not obvious to show that inner ideals of an arbitrary structurable algebra of skew-dimension one correspond to inner ideals of its TKK Lie algebra. (We only prove this for a restricted class of structurable algebras of skew-dimension one.) 
For arbitrary structurable algebras (not necessarily of skew-dimension one), it seems likely that a more restrictive definition than the one of Garibaldi will be needed, but we have not pursued this.

Recently, some deep connections between structurable algebras and low rank geometries have been discovered.
More precisely, Lien Boelaert, Tom De Medts and Anastasia Stavrova have established a strong relationship between structurable division algebras and Moufang sets \cite{Boelaert2019}.
In \cref{sec 4,sec 5}, we translate this result to the setting of inner ideals of Lie algebras of relative rank one.
This settles the case for rank $1$ geometries. 

For rank $2$ geometries, Boelaert and De Medts made two different connections between structurable algebras and Moufang quadrangles (of exceptional type) in \cite{Boelaert2013,Boelaert2015}.
In the former paper, the associated structurable algebras have skew-dimension one, whereas in the latter paper, the associated structurable algebras are tensor products of composition algebras.
In this paper, we restrict to the other Moufang polygons, namely triangles and hexagons.
However, we believe that the ideas in this paper can also be used to handle the Moufang quadrangles and to connect the two different constructions in \cite{Boelaert2013} and \cite{Boelaert2015} with each other.
We hope to present those results in a subsequent paper.

\subsection{Organization of the paper}

In \cref{sec 2} we recall the necessary preliminaries about structurable algebras, inner ideals of Lie algebras, Moufang polygons, extremal geometries and Moufang sets. We add a few more elementary results that we will need later.
\cref{sec 3} is the most technical part of this paper. We prove, among other things, that we can reduce the situation to inner ideals containing $\SS_+$. 
We then apply this reduction theorem in all other sections, except in \cref{sec 4}.

We then proceed to prove our main theorem.
We prove \cref{main:1} in \cref{sec 4,sec 5}, we prove \cref{main:2} in \cref{sec 6} and we finally prove \cref{main:3} in \cref{sec 7}.
In each case, we include the description of the root groups in terms of the corresponding structurable algebra.

\subsection{Acknowledgment}

The idea to use inner ideals of Lie algebras has been inspired by a talk of Arjeh Cohen during the workshop ``Buildings and Symmetry'' at the University of Western Australia in September 2017. We thank Alice Devillers, Bernhard M\"uhlherr, James Parkinson and Hendrik Van Maldeghem for organizing this workshop. We are grateful to Arjeh Cohen for sharing his preprint \cite{Cohen2020} with us.

\subsection{Assumptions}

We assume that all our algebras are finite-dimensional over a field $k$ of characteristic different from $2$ and $3$.
Except for Lie algebras, all algebras are assumed to be unital, but not necessarily associative.

\section{Preliminaries}
\label{sec 2}

\subsection{Structurable algebras}

Structurable algebras have been introduced by Bruce Allison in \cite{Allison1978} as a generalization of Jordan algebras. It is precisely in this context that we will study them. Each Jordan algebra gives rise, via the Tits--Kantor--Koecher (TKK) construction, to a Lie algebra equipped with a $3$-grading. This construction has been generalized to structurable algebras, giving rise to Lie algebras equipped with a $5$-grading; see \cref{def:Lie alg} below.

The material in this subsection is based on \cite[Chapter 2]{Boelaert2019} and we refer to that paper for a more detailed exposition.

\begin{definition}\label{def:struct algebra}
	Let $\A$ be a $k$-algebra equipped with an involution $\sigma \colon x \mapsto \overline{x}$, i.e., a $k$-linear map satisfying $\overline{x.y}=\overline{y} .\overline{x}$ for all $x,y \in \A$.
	Let
	\[ V_{x,y}(z):=(x\overline y)z+(z\overline y)x-(z\overline x)y \]
	for all $x,y,z \in \A$.
	If
		\[ [V_{x,y},V_{z,w}]=V_{V_{x,y}(z),w}-V_{z,V_{y,x}(w)} \]
	for all $x,y,z,w\in\A$ (where the left hand side denotes the Lie bracket of the two operators)
	then we call $\A$ a \textit{structurable algebra}.
\end{definition}

\begin{definition}

	Let $\A$ be a structurable algebra; then $\A=\mathcal H\oplus \SS$, with 
		\[ \mathcal H=\{h\in\A\mid\overline h=h\} \text{ and } \SS=\{s\in\A\mid \overline s=-s\}.\]
	The elements of $\mathcal H$ are called \textit{hermitian elements}, the elements of $\SS$ are called \textit{skew elements}. 
	The dimension of $\SS$ is called the \textit{skew-dimension} of $\A$.
	
\end{definition}

\begin{definition}

	Let $\A$ be a structurable algebra. 
	An element $u \in \A$ is called \textit{conjugate invertible} (or simply \textit{invertible}) if there exists an element $\hat u \in \A$ such that
	 	$$ V_{u,\hat u} = \id, \text{ or equivalently, } V_{\hat u,u} = \id. $$
	If $u$ is conjugate invertible, then the element $\hat u$ is uniquely determined, and is called the \textit{conjugate inverse} of $u$.
	Moreover, if $u$ is conjugate invertible, the operator $U_u$ is invertible; see \cite[Section 6]{Allison1981}.
\end{definition}

\begin{definition}
	An \textit{ideal} of $\A$ is a two-sided ideal stabilized by the involution, and $\A$ is \textit{simple} if its only ideals are $\{0\}$ and $\A$.
	The center of $\A$ is defined by
	\[ Z(\A)=\{z\in \mathcal H \mid [z,\A]=[z,\A,\A]=[\A,z,\A]=[\A,\A,z]=0\},\]
	and $\A$ is \textit{central} if its center equals $k1$.
\end{definition}

\begin{definition}\label{def:skewer}
	The following map, called the \textit{skewer map}, plays an important role in the theory of structurable algebras:
	\[ \psi \colon \A\times \A\rightarrow \mathcal S \colon (x,y)\mapsto x\overline y-y\overline x .\]
\end{definition}

\begin{definition}\label{def:epsilon delta}

For each $x \in \A$, we let $L_x$ and $R_x$ denote the left and right multiplication by $x$, respectively. 
For each $A\in\End (\A)$, we define two new $k$-linear operators 
\begin{align*}
	A^\epsilon &:= A-L_{A(1)+\overline{A(1)}}\,, \\
	A^\delta &:= A+R_{\overline{A(1)}}\,.
\end{align*}
One can verify that 
\begin{align}
	V_{x,y}^\epsilon &= -V_{y,x} \quad \text{and} \label{formula:V epsilon} \\ 
	V_{x,y}^\delta (s) &= -\psi (x,sy) \label{formula:V delta}
\end{align}
for all $x,y\in \A$ and $s\in \SS$
and that
\begin{align}
	(L_rL_t)^\epsilon & = -L_tL_r \quad \text{and} \label{formula:Ls epsilon} \\
	(L_rL_t)^\delta(s) & = s(tr)+r(ts) \label{formula:Ls delta}
\end{align}
for all $r,s,t \in \SS$.
\end{definition}

\begin{definition}
	By the definition of a structurable algebra, the subspace
	\[ \Inst(\A) := \operatorname{span}\{V_{x,y}\mid x,y\in \A\}\]
	is a Lie subalgebra of $\End(\A)$.
\end{definition}

We recall the notion of derivations in structurable algebras.

\begin{definition}
	A \textit{derivation} of $\A$ is a $k$-linear map $D \colon \A\rightarrow\A$ such that $D(ab)=D(a)b+aD(b)$ and $D(\overline a)=\overline{D(a)}$, for all $a,b\in \A$. In particular, we have $D(\SS) \subseteq \SS$ and $D(\HH) \subseteq \HH$.
	
	The Lie algebra of all derivations of $\A$ is denoted by $\Der(\A)$. 
\end{definition}

In the following lemma we use the notation $T_x:=V_{x,1}$.
\begin{lemma}[{\cite[page 1840]{Allison1979}}]
\label{S1 Decom Inst}
We have 
\[ \Inst(\A)=\{T_x\mid x\in \A\} \oplus \Inder(\A),\]
with $\Inder(\A)$ a certain Lie subalgebra of $\Der(\A)$ (namely the Lie subalgebra of inner derivations, but we will not need this explicitly).
\end{lemma}

\begin{lemma}
\label{V operator symmetry}
	We have	$V_{a,b}(c)=V_{c,b}(a)+\psi(a,c)b$ for all $a,b,c \in \A$.
	Moreover, we have $V_{a,a} = L_{a\overline{a}} = T_{a\overline{a}}$ for all $a \in \A$.
\end{lemma}
\begin{proof}
    This is immediate from the definitions.
\end{proof}

We are now ready to introduce the TKK construction for structurable algebras.

\begin{definition}\label{def:Lie alg}

Consider two copies $\A_+$ and $\A_-$ of $\A$ with corresponding isomorphisms $\A \to \A_+ \colon a \mapsto a_+$
and $\A \to \A_- \colon a \mapsto a_-$, and let $\SS_+\subset \A_+$ and $\SS_-\subset \A_-$ be the corresponding subspaces of skew elements.
Define the vector space
\[K(\A)=\SS_-\oplus \A_-\oplus \Inst(\A) \oplus \A_+ \oplus \SS_+.\]
As in \cite[\S 3]{Allison1979}, we define a Lie algebra on $K(\A)$ as the unique extension of the Lie algebra on $\Inst(\A)$ satisfying
\begin{alignat*}{2}
    [V, a_+] &:= (Va)_+ \in \A_+ & [V, a_-] &:= (V^\epsilon a)_-\in\A_- \\
    [V, s_+] &:= (V^\delta s)_+\in \SS_+ & [V, s_-] &:= (V^{\epsilon\delta} s)_-\in \SS_- \\[2ex]
    [s_+,a_+] &:= 0 & [s_-,a_-] &:= 0 \\
    [s_+,a_-] &:= (sa)_+\in \A_+ & [s_-,a_+] &:= (sa)_-\in \A_- \\[2ex]
    [a_+,b_-] &:= V_{a,b}\in \Inst(\A) \qquad \\
    [a_+,b_+] &:= \psi(a,b)_+\in \SS_+ &[a_-,b_-] &:= \psi(a,b)_-\in \SS_- \\[2ex]
    [s_+,t_+]&:=0 & [s_-,t_-]&:=0 \\
    [s_+,t_-]&:=L_{s}L_{t}\in \Inst(\A)
\end{alignat*}
for all $a,b \in \A$, $s,t \in \SS$ and $V \in \Inst(\A)$.
\end{definition}

From the definition of the Lie bracket we clearly see that the Lie algebra $K(\A)$ has a $5$-grading given by $K(\A)_j=0$ for all $|j|>2$ and 
\begin{multline*}
	K(\A)_{-2}=\SS_-,\quad K(\A)_{-1}=\A_-,\quad K(\A)_{0}=\Inst(\A),\\
		K(\A)_{1}=\A_+,\quad K(\A)_{2}=\SS_+.
\end{multline*}

We now turn to examples of structurable algebras.

\begin{example}
	The \textit{central simple} structurable algebras have been classified and are usually listed in 6 (non-disjoint) classes.
	We will need the following three classes:
	\begin{enumerate}
		\item The Jordan algebras are precisely the structurable algebras with trivial involution. They have skew-dimension $0$.
		\item The structurable algebras of skew-dimension $1$ form a separate class and have peculiar features. They all arise as \emph{forms} of structurable matrix algebras (see \cref{def:struct matrix,Skew dim 1 form of} below); see \cite{DeMedts2019} for an explicit construction of these algebras and for a recent overview of the theory.
		\item If $C_i$ is a composition algebra over $k$ with standard involution $\sigma_i$, for $i=1,2$, then the $k$-algebra $C_1\otimes_k C_2$, equipped with the involution 
			\[ \overline{\,\cdot\,} =\sigma := \sigma_1\otimes \sigma_2 ,\]
			is a structurable algebra.
			It has skew-dimension $\dim_k C_1 + \dim_k C_2 - 2$.
	\end{enumerate}
\end{example}

\begin{definition}
	\label{def:struct matrix}
	Let $J$ be a Jordan algebra over a field $k$, let $T \colon J \times J \rightarrow k$ be a symmetric bilinear form, let $\times \colon J \times J \rightarrow J$ be a symmetric bilinear map, and let $N \colon J \rightarrow k$ be a cubic form such that one of the following holds:
	\begin{itemize}
		\item $J$ is a cubic Jordan algebra with a non-degenerate admissible form $N$, with basepoint $1$, trace form $T$, and Freudenthal cross product $\times$; see, for instance, \cite[$\S$38]{Knus1998}.
		\item $J$ is a Jordan algebra of a non-degenerate quadratic form $q$ with basepoint $1$, and $T$ is the linearization of $q$. In this case, $N$ and $\times$ are the zero maps.
		\item $J = 0$, and the maps $N$, $T$ and $\times$ are the zero maps. (In this case, $J$ is not unital.)
	\end{itemize}
	Fix a constant $\eta \in  k^\times$. 
	We now define the \emph{structurable matrix algebra} $M(J,\eta)$ as follows. 
	Let
	\[
		\A=\left\{ \begin{pmatrix}
			k_1 & j_1 \\ j_2 & k_2
		\end{pmatrix} \mid k_1, k_2\in k, j_1, j_2\in J \right\}, \]
	and define the multiplication and the involution by the formulae
	\begin{align*}
	    \begin{pmatrix}
			k_1 & j_1 \\ j_2 & k_2
		\end{pmatrix}\begin{pmatrix}
			k'_1 & j'_1 \\ j'_2 & k'_2
		\end{pmatrix}
		&= \begin{pmatrix}
			k_1k'_1+\eta T(j_1,j'_2) & k_1j'_1+k'_2j_1+\eta (j_2\times j'_2) \\ k'_1j_2+k_2j'_2+j_1\times j'_1 & k_2k'_2+\eta T(j_2,j'_1)
		\end{pmatrix}, \\
	    \overline{\begin{pmatrix}
			k_1 & j_1 \\ j_2 & k_2
		\end{pmatrix}}
		&= \begin{pmatrix}
			k_2 & j_1 \\ j_2 & k_1
		\end{pmatrix},
	\end{align*}
	for all $k_1, k_2, k'_1,k'_2\in k$ and $j_1, j_2, j'_1, j'_2\in J$. 
	It is shown in \cite[Section 8.v]{Allison1978} and \cite[Section 4]{Allison1984}	 that $M(J,\eta)$ is a central simple structurable algebra.
\end{definition}

The following proposition relates all structurable algebras of skew-dimension one to these structurable matrix algebras.

\begin{proposition}[{\cite[Proposition 4.5]{Allison1984}}]
	Let $\A$ be a structurable algebra of skew-dimension one with $s_0^2 = \mu 1$. 
	Then $\A$ is isomorphic to a structurable matrix algebra $M(J,\eta)$ if and only if $\mu$ is a square in $k$.
\end{proposition} 

\begin{corollary}
	\label{Skew dim 1 form of}

	Let $\A$ be a structurable algebra of skew-dimension one. 
	Then there exists a field extension $\ell/k$ of degree at most $2$ such that ${\A} \otimes_{k}  {\ell}$ is isomorphic to a structurable matrix algebra over $\ell$.
\end{corollary}

\begin{remark}
    In \cite{DeMedts2019}, the structurable algebras of skew-dimension one are constructed explicitly, either in terms of \textit{hermitian cubic norm structures}, or equivalently, in terms of (ordinary) cubic norm structures equipped with a semilinear self-adjoint autotopy.
\end{remark}

\subsection{Inner ideals}

We now turn to inner ideals of Lie algebras and in particular of the Lie algebras $K(\A)$ arising from structurable algebras through the TKK construction.

\begin{definition}\label{def inner ideal}
Let $\LL$ be a Lie algebra. 
An \textit{inner ideal} of $\LL$ is a subspace $I$ of $\LL$ satisfying $[I,[I,\LL]]\leq I$.
If $I$ is $1$-dimensional, we call any non-zero element of $I$ an \textit{extremal element}.
We call an inner ideal \textit{singular} if all its non-zero elements are extremal elements.
We call an inner ideal \textit{trivial} if is equal to $0$, and we call it \textit{proper} if it is not equal to $\LL$.
\end{definition}

\begin{definition}
\label{definition zero divisor}
An element $x$ of a Lie algebra $\LL$ is called an \textit{absolute zero divisor}\footnote{In some papers, e.g.\@~\cite{Cohen2006}, these elements are called \textit{sandwich elements}.} if $[x,[x,\LL]]=0$.
A Lie algebra is \textit{non-degenerate} if it has no non-trivial absolute zero divisors. 
\end{definition}

\begin{definition}
Set $\LL_i=K(\A)_i$ for all $i$ with $-2\leq i\leq 2$. 
The \textit{$i$-component} of $x\in\LL$ is the image of the projection of $x$ onto $\LL_i$.
\end{definition}

Recall that we assume $\Char(k) \neq 2,3$ throughout.

\begin{definition}
	 Let $\LL$ be a finite-dimensional 5-graded Lie algebra over $k$. We say that $\LL$ is \textit{algebraic} if for any $(x,s) \in \LL_{\sigma 1} \oplus \LL_{\sigma 2}$ (with $\sigma \in \{ +,- \}$), the endomorphism $\exp(\ad(x+s))$ of $\LL$ is a Lie algebra automorphism. 
	 We say that a structurable algebra $\A$ over $k$ is \textit{algebraic} if $K(\A)$ is algebraic in the above sense.
\end{definition}

The next lemma is mentioned in (the proof of) \cite[Theorem 3.1(1)]{Garcia2011}.

\begin{lemma}
\label{ideals 5 grading}
	Let $\LL=\LL_{-2}\oplus \LL_{-1}\oplus \LL_0\oplus \LL_1\oplus \LL_2$ be a $5$-graded Lie algebra with $\LL_0=[\LL_1,\LL_{-1}]$. Suppose that $I$ is a subspace closed under all Lie algebra endomorphisms $\exp (\ad (l))$, with $l\in \LL_i, i\neq 0$. Then $I$ is an ideal of $\LL$.
\end{lemma}
\begin{proof}
	Let $l\in \LL_i, i\neq 0$, be arbitrary. Consider $x\in I$. By assumption, $I$ contains 
	\[ \exp(\ad (l))(x)-\exp(\ad (-l))(x)=2[l,x]+\frac 1 3 [l,[l,[l,x]]] .\]
	If we replace $l$ by $2l$ in this expression, we get that $I$ contains $4[l,x]+\frac 8 3 [l,[l,[l,x]]]$ as well. Therefore, $I$ contains $12[l,x]$ and hence also $[l,x]$. 
	In particular, $[\LL_i,x]\leq I$ for each $i\neq 0$. 
	
	Next, let $l\in \LL_1$ and $l'\in \LL_{-1}$. Then 
	 \[ [[l,l'],x]=-[[l',x],l]-[[x,l],l']\in I .\]
	Since $\LL_0=[\LL_1,\LL_{-1}]$, this implies $[\LL,x]\leq I$. 
	Hence $I$ is an ideal of $\LL$.
\end{proof}

\begin{theorem}[{\cite[Theorem 2.13, Theorem 3.4]{Stavrova2017}}]
\label{csa is algebraic}
	Any central simple finite\dash dimensional structurable algebra over $k$ is algebraic.
\end{theorem}

The next corollary is similar to \cite[Theorem 3.1.(1)]{Garcia2011}. (Notice, however, that the characteristic $5$ case is excluded in that paper).

\begin{corollary}
\label{A nondeg}
	Let $\A$ be a central simple structurable algebra over $k$. Then $K(\A)$ is a non-degenerate Lie algebra.
\end{corollary}
\begin{proof}
	$\LL:=K(\A)$ is a central simple $5$-graded Lie algebra such that $\LL_0=[\LL_1,\LL_{-1}]$ (see \cite[\S 5]{Allison1979}). 
	Moreover, by definition of algebraicity of $\A$ and \cref{csa is algebraic}, all endomorphisms $\exp (\ad (l))$ of $\LL$, with $l\in \LL_i, i\neq 0$, are actually automorphisms of $\LL$. 
	Let $I$ be the subspace spanned by all absolute zero divisors. It is clear that $I$ is closed under any Lie algebra automorphism.
	By \cref{ideals 5 grading}, $I$ is an ideal. 
	Since $\LL$ is simple, $I=\LL$ or $I=0$.
	If $I=\LL$, then $\LL$ is generated by (a finite set of) absolute zero divisors, and is thus is nilpotent (see \cite{Zelmanov1980}), which is impossible since $\LL$ is non-abelian and simple. 
\end{proof}

\begin{corollary}
\label{Inner abelian}
	Let $\A$ be a central simple structurable algebra over $k$. Then any proper inner ideal of $K(\A)$ is abelian. 
\end{corollary}
\begin{proof}
	Since $K(\A)$ is a finite-dimensional non-degenerate simple Lie algebra, this follows from \cite[Lemma 1.13]{Benkart1977}.
\end{proof}

The following result from \cite{Benkart1977} will be useful in the proof of \cref{0 -1 -2} below.

\begin{lemma}[{\cite[Lemma 1.8]{Benkart1977},\cite[Lemma 1.11]{Draper2008}}]
\label{Benkart Lemma}
	If $I$ is an inner ideal of a Lie algebra $\LL$ and $x\in \LL$ an element element such that $\ad_x^3=0$, then $[x,[x,I]]$ is an inner ideal.
\end{lemma}

We will now define some subgroups of $\mathrm{GL}_k(\mathfrak g)$. 
In \cite{Allison1999}, the action of $\mathrm{GL}_k(\mathfrak g)$ on $\mathfrak g$ is denoted on the left, whereas we need an action on the right in order to be compatible with the conventions in the theory of Moufang sets. 
This is why some formulas differ slightly from the ones in \cite{Allison1999}.

\begin{definition}[{\cite[Lemma 3.2.7]{Boelaert2019}}]
\label{convention multiplication}
    Consider $\LL := K(\A)$, with $\A$ a central simple structurable algebra. 
	Let $E_\sigma(\A)$ denote the subgroup of $\Aut(\LL)$ consisting of the automorphisms\footnote{By \cref{csa is algebraic}, these are indeed automorphisms.}
	\[ e_\sigma(a,s):=\exp(\ad(a_\sigma+s_\sigma)), \] 
	and 
	\[ e_\sigma(a,s)e_\sigma(b,t)=e_\sigma \bigl( a+b, s+t+\tfrac{1}{2} \psi(a,b) \bigr) = e_\sigma(b,t)\circ e_\sigma(a,s)\]
	as multiplication, with $a,b\in\A$, $s,t\in\SS$, $\sigma\in \{+,-\}$ arbitrary. 
	In order to be compatible with the above multiplication, we define the multiplication in $\Aut(K(\A))$ as $fg=g\circ f$ for any $f$, $g\in\Aut(K(\A))$. 
	Let $E(\A)$ be the subgroup of $\Aut(\LL)$ generated by $E_+(\A)$ and $E_-(\A)$.
\end{definition}

The explicit computations in \cref{Image e_+,V equiv with V^epsilon,V^2 equiv with (V^epsilon)^2,a invertible V_{a,sa}} will be needed later.

\begin{lemma}
\label{Image e_+}
	Let $a\in\A, s,t\in\SS$. Then
	\begin{align*}
    	e_+(a,s)(t_-) &= t_-+(-ta)_-+(L_sL_t-\tfrac{1}{2}V_{a,ta})+(-s(ta)+\tfrac{1}{6}U_a(ta))_+ \\
    	   & \hspace*{20ex} + (-s(ts)-\tfrac{1}{2}\psi (a,s(ta))+\tfrac{1}{24}\psi (a,U_a(ta)))_+ \, , \\
    	e_-(a,s)(t_+) &= t_++(-ta)_++(-L_tL_s+\tfrac{1}{2}V_{ta,a})+(-s(ta)+\tfrac{1}{6}U_a(ta))_- \\
    	   & \hspace*{20ex} + (-s(ts)-\tfrac{1}{2}\psi (a,s(ta))+\tfrac{1}{24}\psi (a,U_a(ta)))_- \, .
	\end{align*}
\end{lemma}
\begin{proof}
	Using \cref{formula:Ls delta,formula:V delta}, we compute that
	\begin{align*} 
	\text{ad}(a_++s_+)(t_-)&=(-ta)_-+L_sL_t\\
	\text{ad}(a_++s_+)^2(t_-)&=[a_++s_+,\text{ad}(a_++s_+)(t_-)]\\&= -V_{a,ta}+(-2s(ta))_++(-(L_sL_t)^\delta (s))_+\\
	                         &=-V_{a,ta}+(-2s(ta))_++(-2s(ts))_+\\
	\text{ad}(a_++s_+)^3(t_-)&= (V_{a,ta}(a)_+-2\psi (a,s(ta))_+)-\psi(a,s(ta))_+\\&=U_a(ta)_+-3\psi (a,s(ta))_+\\
	\text{ad}(a_++s_+)^4(t_-)&=\psi (a,U_a(ta))_+ .
	\end{align*}
	These identities together with the definition of $e_+(a,s)$ yield the result. The computations for $e_-(a,s)(t_+)$ are similar, using \cref{formula:Ls epsilon}.
\end{proof}

\begin{lemma}
\label{V equiv with V^epsilon}
	Let $a,b\in \A$ such that $a$ is conjugate invertible and let $V\in \Inst (\A)$ such that $V^2=0$ and $(V^\epsilon)^2=0$. Then
	\[ V_{a,b}=0 \iff b=0 \iff V_{b,a}=0 \]
	and 
	\[ V(a)=0 \iff V^\epsilon (\hat a)=0 .\]
\end{lemma}
\begin{proof}
	The first claim follows from $U_a (b)=V_{a,b}(a)=0$ and the fact that $U_a$ is invertible. The second claim follows immediately by $V_{a,b}^\epsilon =-V_{b,a}$. The last claim follows from 
	\[ 0=[V,\id]=[V,V_{a,\hat a}]=V_{V(a),\hat a}+V_{a,V^\epsilon (\hat a)}\]
	and the previous claims.
\end{proof}

\begin{lemma}
\label{V^2 equiv with (V^epsilon)^2}
	If $V, W\in\Inst (\A)$ satisfy $V^\delta(s)=0=W^\delta (s)$ for some conjugate invertible $s\in\SS$, then we have 
	\[ WV=0\iff W^\epsilon V^\epsilon=0 .\]
\end{lemma}
\begin{proof}
	By \cite[(1.10)]{Allison1984} we have $V(sx)=V^\delta(s)x+sV^\epsilon(x)=sV^\epsilon (x)$ and hence $(WV)(sx)=s(W^\epsilon V^\epsilon)(x)$, for any $x\in\A$.
	Since $s$ is conjugate invertible we get the desired equivalence.
\end{proof}

\begin{lemma}
\label{a invertible V_{a,sa}}	
	If $a\in\A$ and $s\in\SS$ are conjugate invertible, then $V_{a,sa}\neq 0$.
\end{lemma}
\begin{proof}
	If $V_{a,sa}=0$, then $0=V_{a,sa}(a)=U_a L_s(a)$.
	Since $a$ and $s$ are conjugate invertible, the operators $U_a$ and $L_s$ are invertible; hence $a=0$, a contradiction.
\end{proof}

\subsection{Moufang polygons and extremal geometries}

In this subsection we state the necessary preliminaries regarding generalized polygons and extremal geometries.

\begin{definition}
	Let $\Gamma=(\mathbf P,\mathbf L,I)$ be a point-line geometry. 
	We call it a \emph{partial linear space} if two distinct points lie on at most one line.
	Assume now that $\Gamma$ is a partial linear space.
	We call $X\subseteq \mathbf P$ a \textit{subspace} if for all $x, y\in X$ with $x$ and $y$ on a (necessarily unique) line $l$, all points of $l$ are contained in $X$.
	We call this subspace \textit{singular} if any two distinct points of this subspace are collinear.
\end{definition}

\begin{definition}
	Consider $n\in\mathbb N\backslash \{0,1,2\}$.
	A \textit{generalized $n$-gon} is a partial linear space $\mathcal T=(\mathbf P,\mathbf L,I)$ such that: 
	\begin{itemize}
		\item $\mathcal T$ does not contain ordinary $m$-gons as subgeometries, for $m<n$;
		\item For any $x,y\in\mathbf P\cup\mathbf L$, there exists a subgeometry containing both $x$ and $y$ isomorphic to an ordinary $n$-gon.
	\end{itemize}
	We call a point-line geometry a \textit{generalized polygon} if it is a generalized $n$-gon for some $n\geq 3$.
	Generalized $3$-, $4$- and $6$-gons are also called \textit{generalized triangles, quadrangles} and \textit{hexagons}, respectively.
\end{definition}

\begin{definition}
	If $\mathcal T=(\mathbf P,\mathbf L,I)$ is a generalized $n$-gon, then we denote its incidence graph by $\tilde{\mathcal T}$. 
	This is the bipartite graph with $\mathbf P\cup\mathbf L$ as points and $\{p,l\}$ as edges, where $p\in\mathbf P$ and $l\in\mathbf L$ with $p\in l$.
	Note that $\tilde{\mathcal T}$ has diameter $n$ and girth~$2n$.
	If $(x_0,\dots,x_n)$ is a path of length $n$ in $\tilde{\mathcal T}$, then we define the subgroup
		$$ U(x_0,\dots,x_n)$$
	of $\Aut(\tilde{\mathcal T})$ to be the pointwise stabilizer in $\tilde{\mathcal T}$ of all vertices at distance $\leq 1$ from $x_1$, $x_2$, $\dots$, or $x_{n-1}$ and we call this a \emph{root group}.
\end{definition}

\begin{notation}
\label{Notation root groups}
	If we fix a cycle $(x_0,\dots,x_{2n-1},x_{2n})$ of length $2n$ in $\tilde{\mathcal T}$, then we set 
	\[ U_i=U(x_{i},\dots,x_{n+i}),\]
	with $1\leq i\leq n$.
\end{notation}

\begin{definition}
	We call a generalized $n$-gon $\mathcal T=(\mathbf P,\mathbf L,I)$ \emph{Moufang} if the root group $U(x_0,\dots,x_n)$ acts transitively on the set of all neighbors of $x_0$ distinct from~$x_1$, for all paths $(x_0,\dots,x_n)$ of length $n$ in $\tilde{\mathcal T}$.
\end{definition}

\begin{lemma}[{\cite[(3.7)]{Tits2002}}]
\label{Moufang general remark}
	Let $\Gamma$ be a generalized polygon, $G$ its automorphism group and $(x_0,\ldots,x_n)$ an $n$-path. Consider neighbors $x,y$ of $x_0$, distinct from $x_{1}$. Then there is at most one element of $U:=U(x_0,\ldots,x_n)$ mapping $x$ on to $y$. In particular, $\Gamma$ is Moufang if and only if $U$ acts sharply transitively on the set of neighbors of $x_0$ distinct from $x_{1}$ (for all $n$-paths).
\end{lemma}

We now turn to extremal geometries.

\begin{definition}[{\cite[p.\@~435]{Cohen2006}}]\label{E_i}
	Consider a Lie algebra $\LL$. Let $E=E(\LL)$ denote the set of extremal elements in $\LL$. Consider $x,y\in E$. Then we say that $x$ and $y$ are
	\begin{itemize}
		\item \textit{identical} if $\langle x\rangle=\langle y\rangle$. We set $(x,y)\in E_{-2}$;
		\item \textit{strongly commuting} if $[x,y]=0$, $\langle x\rangle\neq \langle y\rangle$ and $\lambda x+\mu y$ is either $0$ or an extremal element, for all $\lambda, \mu\in k$. We set $(x,y)\in E_{-1}$;
		\item \textit{commuting} if $[x,y]=0$ and $x$ and $y$ are not strongly commuting or identical. We set $(x,y)\in E_{0}$;
		\item \textit{special} if $[x,y]\neq 0$ and $[x,[x,y]]=0$. We set $(x,y)\in E_{1}$;
		\item \textit{hyperbolic} if $[x,[x,y]]\neq 0$. We set $(x,y)\in E_{2}$.
	\end{itemize}
	Clearly $E\times E=E_{-2}\sqcup E_{-1}\sqcup E_{0}\sqcup E_{1}\sqcup E_{2}$. We set $E_{\leq i}=\bigcup\limits_{j=-2}^{i} E_j$.
\end{definition}

\begin{definition}\label{extremal geometry}
	Let $\LL$ be a Lie algebra. The \textit{associated extremal geometry} is the incidence geometry with as points 
	\[ \E=\E(\LL):=\{\langle x\rangle \mid x\in E=E(\LL)\},\]
	as lines 
	\[ \F=\F(\LL):=\{ \langle x,y\rangle\mid x,y\in E, (x,y)\in E_{-1}\}, \]
	and inclusion as incidence relation.
	We denote this geometry by $\Gamma(\LL)$.
\end{definition}

Recall the definition of an absolute zero divisor (\cref{definition zero divisor}).

\begin{theorem}
\label{Extremal geometry characterisation hexagon}	
	Assume that the finite-dimensional simple Lie algebra $\LL$ has no absolute zero divisors, that $\LL$ is generated by $E(\LL)$, and that $\F(\LL)\neq\emptyset$. If there is a line in $\F(\LL)$ which is a maximal singular subspace of the extremal geometry $(\E,\F)$, then $(\E,\F)$ is a generalized hexagon. 
\end{theorem}
\begin{proof}
	This is \cite[Corollary 18]{Cohen2007} combined with \cite[Theorem 28(ii)]{Cohen2006}.
\end{proof}

\subsection{Moufang sets}

We briefly recall the definition of a Moufang set.
We refer to \cite{DeMedts2009} for a detailed introduction to the subject.

\begin{definition}
Let $X$ be a set (with $|X|\geq 3$) and let $\{U_x\mid x\in X\}$ be a collection of subgroups of $\Sym(X)$. 
The data $(X, \{U_x\}_{x\in X})$ is a \emph{Moufang set} if the following two properties are satisfied:
\begin{itemize}
	\item For each $x\in X$, $U_x$ fixes $x$ and acts  sharply transitively on $X \backslash \{x\}$.
	\item For each $g\in G^+:=\langle U_x\mid x\in X\rangle\leq \Sym(X)$ and each $y\in X$ we have $U_y^g=U_{y.g}$.\footnote{As is common in the theory of Moufang sets, we will always denote group actions on the right.}
\end{itemize}
The group $G^+$ is called the \emph{little projective group} of the Moufang set, and the groups
$U_x$ are called the \emph{root groups}.

Each group $G$ acting doubly sharply transitively on a set $X$ gives rise to a Moufang set (with $U_x = \operatorname{Stab}_G(x)$ and $G^+ = G$). A Moufang set is called \textit{proper} if the action of $G^+$ on $X$ is \textit{not} doubly sharply transitive.
\end{definition}

All known examples of proper Moufang sets with abelian root groups arise from (quadratic) Jordan division algebras \cite{DeMedts2006,DeMedts2008,Gruninger2015}.
More generally, all known examples of proper Moufang sets with (abelian or non-abelian) root groups without elements of order $2$ or $3$ arise from structurable division algebras \cite{Boelaert2019}.
(There are infinite families of counterexamples over fields of characteristic $2$ and $3$, but also those examples are still of algebraic nature.)

\section{Reduction to inner ideals containing $\SS_+$}
\label{sec 3}
In this section, we will provide a tool to reduce our study of inner ideals of $K(\A)$ to those containing the $2$-component $\SS_+$.
We will need one technical condition (\cref{ass:V0}), but as we will see, this assumption will be satisfied in all situations we are interested in; see \cref{triangle Assumption (i),Hexagon Assumption (i)} below.

\begin{assumption}
\label{ass:skew}
    Throughout this section, we will assume that $\A$ is a central simple structurable algebra over $k$ such that $\SS \neq 0$ and \textit{all non-zero elements of $\SS$ are conjugate invertible}.
\end{assumption}

\begin{remark}
	Recall that by \cref{Inner abelian}, all proper inner ideals of $K(\A)$ are abelian.
	We will use this fact repeatedly in this section without explicitly mentioning it.
\end{remark}

\begin{remark}
	It will be obvious that if we replace $\SS_+$ by $\SS_-$ and $\A_+$ by $\A_-$ in any statement in this section, the statement remains valid.
\end{remark}

\begin{lemma}
\label{psi non-degenerate}
	The skewer map $\psi$ is non-degenerate.
\end{lemma}
\begin{proof}
	This is precisely \cite[Lemma 2.2]{Allison1984}. Although it is only stated for structurable algebras of skew-dimension one, it continues to hold for arbitrary central simple structurable algebras with $\SS\neq 0$; this follows from the remarks regarding the \textit{multiplication algebra} on page $189$ of \textit{loc.\@~cit}.
\end{proof}

\begin{lemma}
\label{Inner containing s}
	If $I$ is an inner ideal of $K(\A)$ containing $s_+$, with $s\in\SS$ non-zero, then $\SS_+\leq I$.
\end{lemma}
\begin{proof}
	Note that $I$ contains $[s_+,[s_+,t_-]]=-2sts_+$ for any $t\in\SS$, using \cref{formula:Ls delta}. Since $s$ is conjugate invertible, we get $\SS_+\leq I$.
\end{proof}

Our goal is to show that any proper non-trivial inner ideal of $\LL:=K(\A)$ can be mapped by an element of $E(\A)$ to an inner ideal containing the set of skew elements $\SS_+$.

\begin{lemma}
\label{-1 0 1}
	Let $I$ be a non-trivial inner ideal of $\LL$. Then there exists an automorphism in $E(\A)$ mapping $I$ to an inner ideal containing an element with non-zero $2$-component.
\end{lemma}
\begin{proof}
	By \cref{ideals 5 grading} the subspace $\sum_{\varphi\in E(\A)}\varphi(I)$ is a non-zero ideal of $\LL$.
	Consider $0\neq s\in \SS$ arbitrary, since $\LL$ is non-degenerate, $[s_+,[s_+,\varphi(I)]]\neq 0$ for some $\varphi\in E(\A)$.
	This implies that $\varphi(I)$ is an inner ideal containing an element with non-zero $2$-component.
\end{proof}

\begin{lemma}
\label{0 -1 -2}
	Let $I$ be a minimal inner ideal of $\LL$ containing an element with non-zero $2$-component.
	Then $I\cap (\LL_{-2}\cup \LL_{-1}\cup\LL_0\cup \LL_1)=0$.
\end{lemma}
\begin{proof}
	By assumption $I$ contains an element $x$ with non-zero $2$-component $s_+\in \SS_+$.
	By \cref{Benkart Lemma}, $[x,[x,\SS_-]]\leq I$ is an inner ideal.
	Now note that the $2$-component of $[x,[x,t_-]]$ equals $-2sts$, which is non-zero if $t\in\SS$ is non-zero.
	Hence $I=[x,[x,\SS_-]]$ and $I\cap (\LL_{-2}\cup \LL_{-1}\cup\LL_0\cup \LL_1)=0$.
\end{proof}

\begin{lemma}
\label{-2 -1 0 1 2}
	Let $I$ be a minimal inner ideal of $\LL$ containing an element with non-zero $2$-component. Then there exists an element of $E_{-}(\A)$ mapping $I$ to $\SS_+$.
\end{lemma}
\begin{proof}
	By assumption, $I$~contains an element $x := t_- + b_- + V + a_+ + s_+$, with $a,b\in \A$, $V\in\Inst(\A)$, $s,t\in \SS$, with $s \neq 0$.
	Then $I$ also contains 
	\begin{align*}
		[x,[x,-\hat s_-]] &= [x, -V^{\epsilon\delta} (\hat s)_- + (\hat s a)_- + \id] \\
				   &= \bigl( -(V^{\epsilon\delta})^2(\hat s)_- + (V^{\epsilon\delta}(\hat s)a)_- - L_sL_{V^{\epsilon\delta}(\hat s)} \bigr) \\
				    & \hspace*{8ex} + \bigl( \psi(b,\hat sa)_- +V^\epsilon(\hat sa)_- + V_{a,\hat sa} - a_+ \bigr) \\
				    & \hspace*{8ex} + \bigl( 2t_- + b_- - a_+ - 2s_+ \bigr) \\
				   &= \bigl( 2t+\psi(b,\hat sa)-(V^{\epsilon\delta})^2(\hat s) \bigr)_- + \bigl( b+V^\epsilon(\hat sa)+V^{\epsilon\delta}(\hat s)a \bigr)_- \\
				    & \hspace*{8ex} + \bigl( V_{a,\hat sa} - L_sL_{V^{\epsilon\delta}(\hat s)} \bigr) - 2a_+ -2s_+ .
	\end{align*}
	By adding twice the element $x$, we find that $I$ contains
	\begin{multline*}
	   y := \bigl( 4t+\psi(b,\hat sa)-(V^{\epsilon\delta})^2(\hat s) \bigr)_- + \bigl( 3b+V^\epsilon(\hat sa)+V^{\epsilon\delta}(\hat s)a \bigr)_- \\
				    + \bigl( 2V + V_{a,\hat sa} - L_sL_{V^{\epsilon\delta}(\hat s)} \bigr) .
	\end{multline*}
	Since $y \in \LL_{-2} \oplus \LL_{-1} \oplus \LL_0$, \cref{0 -1 -2} now implies that $y=0$.
	Set
	\[ s' = -\tfrac{1}{2}V^{\epsilon\delta} (\hat s) \quad \text{and} \quad a'=\hat sa. \]
	Expressing that the $0$-component of $y$ is $0$ yields
	\begin{equation}\label{eq-V}
	   V = \tfrac{1}{2} (-V_{a,\hat sa}+L_sL_{V^{\epsilon\delta}(\hat s)}) = \tfrac{1}{2} V_{sa',a'} - L_sL_{s'},
	\end{equation}
	using $a=-sa'$.
	Now we use this equation, the formulas in \cref{def:epsilon delta} and the fact that the $(-1)$-component of $y$ equals $0$ to get
	\begin{align*}  
		3b &= -V^\epsilon(\hat sa)-V^{\epsilon\delta}(\hat s)a \\
		  &= \bigl( \tfrac{1}{2} V_{a',sa'}(a')-s'(sa') \bigr) + \bigl( -\tfrac{1}{2} \psi(a',\hat s(sa'))-s'(s\hat s)-\hat s(ss') \bigr) a \\
		  &= \tfrac{1}{2} V_{a',sa'}(a')-3s'(sa'),
	\end{align*}
	hence
	\begin{equation}\label{eq-b}
	    b = \tfrac{1}{6} V_{a',sa'}(a') - s'(sa')
	\end{equation}
	Finally, expressing that the $(-2)$-component of $y$ equals $0$ yields, using \cref{eq-V,eq-b}, that
	\begin{align*}
		4t &= -2V^{\epsilon\delta}(s') + \psi (a',b) \\
		  &= -\psi(a',s'(sa'))-4s'(ss') + \psi \bigl(a',\tfrac{1}{6} V_{a',sa'}(a')-s'(sa')\bigr) \\
		  &= -4s'(ss')-2\psi (a',s'(sa')) + \tfrac{1}{6} \psi (a',U_{a'}(sa')).
	\end{align*}
	Hence
	\begin{equation}\label{eq-t}
	   t = -s'(ss')-\tfrac{1}{2} \psi (a',s'(sa')) + \tfrac{1}{24} \psi (a',U_{a'}(sa')) .
	\end{equation}
	By \cref{eq-V,eq-b,eq-t} and \cref{Image e_+}, we conclude that $x= e_-(a',s')(s_+)$. 
	
	Hence $J:=e_-(a',s')^{-1}(I)$ is an inner ideal containing $s_+\in\SS_+$.
	Since $s$ is non-zero and $I$ is minimal, \cref{Inner containing s} implies $\SS_+ =J$.
\end{proof}

\begin{assumption}
\label{ass:V0} 
Consider the following (technical) assumption:
\begin{quote}    
	If $V\in\Inst(\A)$ is such that $V^\delta(\SS)=0$ and $V^2=0$, then there exists some
	$W\in \Inst(\A)$ such that $U:=[V,[V,W]] \in\Inst(\A)$ satisfies $U^2\neq 0$. In particular, $U\not\in\langle V\rangle$.
\end{quote}
We will always explicitly mention when we make this assumption.
\end{assumption}

\begin{remark}
	Our motivation for \cref{ass:V0} is that it basically ensures that we cannot get a generalized quadrangle as the point-line geometry we will obtain in the next sections.
	We now explain this in some detail.
	
	Assume that there are two types of non-trivial proper inner ideals in $K(\A)$, say of dimension $m$ and $n$, with $m<n$, and assume that the point-line space with as points the $m$-dimensional inner ideals and as lines the $n$-dimensional inner ideals is a generalized quadrangle. 
	
	Since every non-zero element of $\SS$ is conjugate invertible, $\SS$ is a minimal inner ideal by \cref{Inner containing s}, hence $m=\dim (\SS)$.
	So $\SS_+$ and $\SS_-$ are two points of the generalized quadrangle. 
	Since $[s_+,\hat s_-]=L_sL_{\hat s}=-\id\neq 0$, these two points are not collinear (since collinear points lie in a common $n$-dimensional inner ideal, which is abelian). 
	Hence they should be at distance $2$ and have a common neighbor, say~$I$.
	Then $I$ is an inner ideal of $K(\A)$ of dimension $m$.
	Since $\SS_+$ and $I$ are collinear, they are contained in a common proper inner ideal, which is abelian, hence $[\SS_+,I]=0$.
	Using that all non-zero elements of $\SS$ are conjugate invertible, we get $I\leq \Inst (\A)\oplus \A_+\oplus \SS_+$.
	Similarly, the fact that $\SS_-$ and $I$ are collinear implies $I\leq \SS_-\oplus \A_-\oplus \Inst (\A)$ and hence $I\leq \Inst (\A)$.
	Consider any non-zero $V\in I$.
	Since $I\leq \Inst (\A)$, $[V,[V,a_+]]=V^2(a)_+$ for all $a \in \A$ implies $V^2=0$. 
	Using $[I,\SS_+]=0$ we get $V^\delta(\SS)=0$.
	
	However, for each $W\in\Inst (\A)$, we now have $U:=[V,[V,W]]\in I$, so by the previous observations $U^2=0$.
	Hence $\A$ does not satisfy \cref{ass:V0}.
\end{remark}

\begin{lemma}
\label{Containing S}
	Let $I$ be a proper inner ideal properly containing $\SS_+$. Then $I = I_0\oplus I_1\oplus \SS_+$ for some inner ideal $I_0\leq \Inst(\A)$ and some subspace $0\neq I_1\leq \A_+$.
	If, in addition, $\A$ satisfies \cref{ass:V0}, then $I_0=0$.
\end{lemma}
\begin{proof}
	Consider an arbitrary element $x := u_- + b_- + V + a_+ + t_+ \in I$.
	Choose some $0 \neq s \in \SS$; then by assumption, $s_+ \in I$.
	Since $I$ is abelian, $[s_+,x] = 0$, which implies $[s_+,b_-] = 0$ and $[s_+,u_-] = 0$.
	Since $s$ is conjugate invertible, $(sb)_+=[s_+,b_-]=0$ implies $b=0$ and
	$L_s L_u = [s_+,u_-] = 0$ implies $u=0$.
	Hence $I \leq \LL_0 \oplus \LL_1 \oplus \LL_2$. 
	
	Note that $a_++2t_+=-[x,\id]=[x,[s_+,\hat s_-]]\in I$.
	Together with $t_+\in I$ and $V+a_++t_+=x\in I$ this implies $V\in I$ and $a_+\in I$.
	Hence $I=I_0\oplus I_1\oplus \SS_+$, with $I_0\leq \Inst(\A)$ and $I_1\leq \A_+$.
	
	Consider $V$ and $W$ in $I_0$ arbitrary. 
	Since $I$ is abelian $[V,\SS_+]=V^\delta(\SS)=0$.
	Moreover, since $I \leq \LL_0 \oplus \LL_1 \oplus \LL_2$, $[V,[W,\LL_i]]=0$ for $i=-1,-2$.
	Hence \cref{V^2 equiv with (V^epsilon)^2} implies $[V,[W,\LL_1]]=0$ and clearly $[V,[W,\LL_0]]\leq I_0$.
	Hence $I_0$ is indeed an inner ideal.
	
	Consider $0\neq V\in I_0$, then $V(sa)_+=[V,[s_+,a_-]]\in I$ for all $a\in \A$ and $0\neq s \in\SS$.
	Since $V\neq 0$ and $s$ is conjugate invertible, this implies $I_1\neq 0$.
	
	Now assume that $\A$ satisfies \cref{ass:V0}.
	Consider $0\neq V\in I_0$.
	Note that the fact that $I$ is abelian implies $V^\delta(\SS)=0$.
	For each $b \in \A$, $(V^\epsilon)^2(b)_-=[V,[V,b_-]]\in I \leq \LL_0 \oplus \LL_1 \oplus \LL_2$ must be $0$; hence $(V^\epsilon)^2=0$.
	By \cref{V^2 equiv with (V^epsilon)^2}, we get $V^2=0$ and by \cref{ass:V0}, there exists $W\in\Inst(\A)$ such that $U:=[V,[V,W]]$ satisfies $U^2\neq 0$. 
	Hence $U\in I_0$.
	The previous argument shows that $U^\delta(\SS) = 0$ and $(U^{\epsilon})^2=0$ and hence, by \cref{V^2 equiv with (V^epsilon)^2}, $U^2=0$, a contradiction. 
	We conclude that $I \leq \LL_1 \oplus \LL_2 = \A_+ \oplus \SS_+$.
\end{proof}

\begin{theorem}
\label{Reducing to S_+}
    Let $\A$ be a central simple structurable algebra satisfying \cref{ass:skew}
    and let $I$ be a proper non-trivial inner ideal of $K(\A)$.
    \begin{enumerate}
        \item There exists an element of $E(\A)$ mapping $I$ to an inner ideal $J$ such that $\SS_+ \leq J \leq \Inst(\A)\oplus \A_+\oplus \SS_+$.
        \item If, in addition, $\A$ satisfies \cref{ass:V0}, then $\SS_+ \leq J \leq \A_+ \oplus \SS_+$.
    \end{enumerate}
\end{theorem}
\begin{proof}
	Let $J$ be a minimal inner ideal contained in $I$.
	By \cref{-1 0 1}, we can assume that $J$ contains an element with a non-zero $2$-component. 
	By \cref{-2 -1 0 1 2}, we may then assume that $J$ equals $\SS_+$. 
	Now \cref{Containing S} concludes the proof.
\end{proof}

The next result will not be used in this paper, but it is a useful fact requiring only very little effort to prove at this point.
\begin{proposition}
\label{Inner spanned by min}
	Let $\A$ be a central simple structurable algebra satisfying \cref{ass:skew}.
	Any proper non-trivial inner ideal of $K(\A)$ is linearly spanned by its minimal inner ideals.
\end{proposition}
\begin{proof}
	We will prove this claim by induction on $\dim (I)$.
	When $\dim(I)=\dim (\SS)$, $I$ itself is a minimal inner ideal by \cref{Reducing to S_+} and then the claim is obvious.
	Assume $\dim(I)>\dim(\SS_+)$.
	By \cref{Reducing to S_+}, we may assume that $\SS_+\leq I$.
	By \cref{Containing S}, $I=I_0\oplus I_1\oplus \SS_+$ for some inner ideal $I_0\leq \Inst(\A)$ and some subspace $0\neq I_1\leq \A_+$.
	By induction, $I_0$ is spanned by its minimal inner ideals (where we allow $I_0 = 0$).
	
	Let $a_+\in I_1$ be arbitrary and let $s \in \SS \setminus \{ 0 \}$.
	Then $[a_+,[a_+,s_-]]=-V_{a,sa}$ is contained in $I_0$ and hence $[V_{a,sa},[a_+,s_-]]=V_{sa,a}(sa)_-\in\LL_{-1}$ is contained in $I$ and is thus $0$.
	Together with \cref{Image e_+} this gives $e_-(sa)(\hat s_+) = -\frac{1}{2}V_{a,sa} + a_+ + \hat s_+$.
	This is an element of the minimal inner ideal $e_-(sa)(\SS_+)$.
	Since $\hat s_+$ is an element of a minimal inner ideal and $V_{a,sa}\in I_0$ is contained in the linear span of minimal inner ideals, $a_+$ is also contained in the linear span of minimal inner ideals. 
\end{proof}

\section{Moufang sets with abelian root groups}
\label{sec 4}

We are now prepared to begin our investigation of the geometry of proper non-trivial inner ideals in specific situations. In \cref{sec 4,sec 5}, we will deal with the case of structurable \textit{division} algebras and we will show that this gives rise to Moufang sets.
The case of Jordan division algebras will give rise to Moufang sets with abelian root groups (\cref{sec 4}) whereas the structurable division algebras of skew-dimension $>0$ will give rise to Moufang sets with non-abelian root groups (\cref{sec 5}).

Throughout this section, we will assume that $\A$ is a Jordan division algebra and we denote $\A$ by $J$. 
We also write $a^{-1} := \hat a$, as is usual in Jordan theory.
Note that $J$ is commutative and hence $V_{x,y}z=V_{z,y}x$ for all $x,y,z \in J$.
Set $\LL=K(J)$.
\begin{remark}
	It is obvious, as in the previous section, that if we replace $J_+$ by $J_-$ in any statement in this section, the statement remains valid.
\end{remark}

\begin{lemma}
\label{JordanInner}
	Let $I$ be an inner ideal of $\LL$. 
	If $I\cap J_+\neq 0$, then $J_+\leq I$.
\end{lemma}
\begin{proof}
	Consider $0\neq a_+\in I\cap J_+$. 
	The inner ideal $I$ contains
	\[ [a_+,[a_+,b_-]] = [a_+,V_{a,b}] = -(V_{a,b}a)_+ = -(U_a b)_+, \]
	for every $b\in J$.
	Since $J$ is a Jordan division algebra, $U_a$ is invertible and hence this implies $J_+\leq I$.
\end{proof}

\begin{lemma}
\label{V zero}
	Let $x,y\in J \setminus \{ 0 \}$. 
	Then $V_{x,y}^2\neq 0$ and $(V_{x,y}^\epsilon)^2 \neq 0$.
\end{lemma}
\begin{proof}
	Since $J$ is a division algebra, the elements $x$ and $y$ are invertible.
	Then $V_{x,y}(y^{-1})=V_{y^{-1},y}(x)=x$ and hence $(V_{x,y})^2(y^{-1})=V_{x,y}(x)=U_x(y)$ which is non-zero since $U_x$ is invertible; therefore $V_{x,y}^2 \neq 0$.
	Since $V_{x,y}^\epsilon=-V_{y,x}$, we also have $(V_{x,y}^\epsilon)^2 \neq 0$.
\end{proof}

\begin{lemma}
\label{I contains Jordans}
	Let $I$ be an inner ideal of $\LL$ with $J_-\oplus J_+\leq I$. 
	Then $I=\LL$.
\end{lemma}
\begin{proof}
	Consider $x,y\in J$ arbitrary.
	Then $I$ contains
	\[ [y_-,[x_+,\id]]=[x_+,y_-]=V_{x,y}.\]
	Hence $\Inst (J)\leq I$ and our claim follows.
\end{proof}

\begin{lemma}
\label{Inner0}
	Let $0\neq V\in  \Inst(J)$ be arbitrary. 
	The only inner ideal of $\LL$ containing $V$ is $\LL$ itself.
\end{lemma}
\begin{proof}
	Note that $V\neq 0$ implies $V^\epsilon \neq 0$.
	Suppose that $I$ is an inner ideal containing $V$.
	For each $a \in J$, we then have
	\[ I \ni [V,[V,a_+]] = [V,V(a)_+] = V^2(a)_+ . \]
	Similarly, $I$ also contains $(V^{\epsilon})^2(a)_-$, for all $a\in J$.
	We distinguish the following cases:
	\begin{itemize}
		\item $V^2\neq 0\neq (V^{\epsilon})^2$: By \cref{JordanInner} we get $J_-\oplus J_+\leq I$. 
		\item $V^2\neq 0 = (V^\epsilon)^2$: By \cref{JordanInner} we get $J_+\leq I$. 
				So $I$ contains $x_+$ and $V$ and thus
				\[ [x_+,[V,x^{-1}_-]]=[x_+,V^\epsilon(x^{-1})_-]=V_{x,V^\epsilon(x^{-1})}, \]
			for any $0\neq x\in J$. 
			By \cref{V zero} there exists an $x\in J$ such that $V_{x,V^\epsilon(x^{-1})}\neq 0$.
			By the first case and \cref{V zero} we get $J_-\oplus J_+\leq I$. 
		\item $V^2= 0 \neq (V^\epsilon)^2$: Similarly as in the previous case we get $J_-\oplus J_+\leq I$.
		\item $V^2=0=(V^\epsilon)^2$: For any $x,y\in J$ the inner ideal $I$ contains 
		 \[ [V,[V,V_{x,y}]]=[V,V_{V(x),y}+V_{x,V^{\epsilon}(y)}]=2V_{V(x),V^\epsilon(y)}.\] 
		 By $V\neq0\neq V^\epsilon$ and \cref{V zero} we get that $0\neq V_{a,b}\in I$ for some $a,b\in J$.
		 Again, by the first case and \cref{V zero} we get $J_-\oplus J_+\leq I$.
	\end{itemize}
	So, in any case, we get $J_-\oplus J_+\leq I$.
	Hence, by \cref{I contains Jordans} we have $I=\LL$.
\end{proof}

\begin{lemma}
\label{Inner01}
	Let $I$ be a proper inner ideal containing $a_++V$, where $V\neq 0, a\in J$.
	Then $I=\LL$.
\end{lemma}
\begin{proof}
	If $a=0$, the claim follows from \cref{Inner0}.
	Note that $I$ contains
	\[ [V+a_+,[V+a_+,\id]]=[V+a_+,-a_+]=(-V(a))_+.\]
	Suppose $V(a)\neq 0$, we get by \cref{JordanInner} that $J_+\leq I$ and hence $0\neq V\in I$ and thus, by \cref{Inner0}, $I=\LL$.
	So we may assume $V(a)=0$. 
	By \cref{V equiv with V^epsilon} we get $V^\epsilon(a^{-1})=0$.
	Hence, $I$ contains
	\[ [V+a_+,[V+a_+,a^{-1}_-]]=[V+a_+,V^\epsilon(a^{-1})_-+\id]=-a_+, \]
	and thus, as in the previous case, $J_+\leq I$ and $I=\LL$.
	\end{proof}

\begin{corollary}
\label{Inner containing J}
	There are only two inner ideals of $\LL$ containing $J_+$, namely $J_+$ and $\LL$.
\end{corollary}
\begin{proof}
	Let $I$ be an inner ideal containing $J_+$.
	If $J_-\cap I\neq 0$ we get $J_-\leq I$ by \cref{JordanInner} and \cref{I contains Jordans} yields $I=\LL$.
	If $J_-\cap I=0$ and $I\neq J_+$, $I$ contains an element $a_-+V$ with $V\neq 0$ and \cref{Inner01} yields $I=\LL$, a contradiction.
\end{proof}

\begin{proposition}
\label{Set Inner ideals}
	Let $I$ be a proper non-trivial inner ideal of $\LL$ with $I \neq J_+$.
	Then there exists a unique $x\in J$ such that
	\[ I = e_+(x)(J_-) = \bigl\{ b_-+V_{x,b}-\tfrac{1}{2}V_{x,b}(x)_+ \mid b\in J \bigr\} . \] 
	Moreover, for each $x \in J$, $e_+(x)(J_-)$ is an inner ideal.
\end{proposition}
\begin{proof}
	Since $J_-$ is clearly an inner ideal of $\LL$, we get that $e_+(x)(J_-)$ is an inner ideal for all $x\in J$, using $e_+(x)\in\Aut (\LL)$.
	
	Consider an arbitrary proper non-trivial inner ideal $I\neq J_+$ of $\LL$ and let $0\neq b_-+V+a_+\in I$ be arbitrary.
	If $b=0$, then \cref{Inner01} yields $V=0$.
	But this is also impossible since $I$ then contains $J_+$ and is, by \cref{Inner containing J}, equal to $\LL$.
	So $b\neq 0$ and $I$ contains
	\begin{align*}
	 [b_-+V+a_+,[b_-+V+a_+,b^{-1}_+]]&=[b_-+V+a_+,-\id+V(b^{-1})_+]\\
	 &=-b_--V_{V(b^{-1}),b}+(V^2(b^{-1})+a)_+ .
	\end{align*}
	Hence, $I$ contains $W+c_+:=(V-V_{V(b^{-1}),b})+(V^2(b^{-1})+2a)_+$.
	If $W=0$ and $c\neq 0$, \cref{JordanInner} and \cref{Inner containing J} imply $I=\LL$, a contradiction.
	If $W\neq 0$, then \cref{Inner01} implies $I=\LL$ which yields a contradiction.
	Hence $W=0$ and $c=0$ and thus $V=V_{V(b^{-1}),b}$ and 
	\[
	a=-\tfrac{1}{2}V_{V(b^{-1}),b}(V(b^{-1})) .
	\]
	Hence $(e_+(V(b^{-1})))^{-1}(I)\cap J_- \neq 0$ and by \cref{Inner containing J,JordanInner}, we get $(e_+(V(b^{-1})))^{-1}(I)=J_-$, or equivalently, $I=e_+(V(b^{-1}))(J_-)$.
	
	The uniqueness claim follows since $V_{x,b}=V_{y,b}$ implies $V_{x-y,b}=0$ and hence $x=y$, by \cref{V zero}.
\end{proof}

\begin{theorem}
\label{Set Inner are Moufang}
	Let $J$ be a Jordan division algebra. Then the set of all proper non-trivial inner ideals of $\LL=K(J)$ forms a Moufang set, with root groups 
	\begin{align*} 
		U_{J_+}&=E_+(J); \\
		U_{e_+(j)(J_-)}&=E_-(J)^{e_+(j)}, \quad \text{for all } j\in J .
	\end{align*}
	This Moufang set is isomorphic to the Moufang set $\mathbb{M}(J)$ as defined in \cite[\S 4]{DeMedts2006}.
\end{theorem}
\begin{proof}
	Let $X$ be the set of proper non-trivial inner ideals.
	By \cref{Set Inner ideals} we see that $U_{J_+}$ acts sharply transitively on the set of all proper non-trivial inner ideals different from $J_+$.
	Similarly, $U_{J_-}$ acts sharply transitively on the set of all proper non-trivial inner ideals different from $J_-$. 
	Clearly $E_-(J)^{e_+(j)}$ fixes $e_+(j)(J_-)$ and acts sharply transitively on the set of all other proper non-trivial inner ideals, for $j\in J$ arbitrary.\footnote{Recall \cref{convention multiplication}.}
	By definition of the root groups, we have $G^+=\langle E_-(J),E_+(J)\rangle$.
	So in order to prove $U_x^g=U_{x.g}$ for all $g\in G^+$ and $x\in X$ it suffices to show this for all $g\in E_-(J)$ and $x\in X$, since it is clear for all $g\in E_+(J)$ by construction.
	By the last equation in \cite[Theorem 5.1.1]{Boelaert2019} we find, for each $0\neq x\in J$, an element $y\in J$ such that $U_{J_-}^{e_+(x)}=U_{J_+}^{e_-(y)}$ .
	Since $U_{J_+}^{e_-(y)}$ fixes $e_-(y)(J_+)$ we get $U_{e_-(y)(J_+)}=U_{J_-}^{e_+(x)}=U_{J_+}^{e_-(y)}$.
	Now it is clear that $U_x^g=U_{x.g}$ for all $g\in E_-(J)$ and $x\in X$ as well.
	The last claim now follows since the corresponding abstract rank one groups coincide; see \cite[Lemma 1.1.12 and \S 5.1]{Boelaert2019}.
\end{proof}

\section{Moufang sets with non-abelian root groups}
\label{sec 5}

We continue with our investigation of structurable division algebras.
Throughout this section, $\A$ is a central simple structurable division algebra with $\SS\neq 0$.
Notice that \cref{ass:skew} is trivially satisfied.
Set $\LL=K(\A)$.

\begin{lemma}
\label{Set Assumption}
	$\A$ satisfies \cref{ass:V0}.
\end{lemma}
\begin{proof}
	Consider $0\neq V\in\Inst(\A)$ arbitrary with $V^2=0$ and $V^\delta (\SS)=0$.
	By \cref{V^2 equiv with (V^epsilon)^2} we get $(V^\epsilon)^2=0$.
	Consider $a\in\A$ arbitrary, with $a\neq 0$.
	Then, using $V^2(a)=0$ and $(V^{\epsilon})^2(\hat a)=0$, we get  
	\[ 0=[V,[V,\id]]=[V,[V,V_{a,\hat a}]]=2V_{V(a),V^\epsilon(\hat a)}.\]
	Since $\A$ is division, \cref{V equiv with V^epsilon} implies $V(a)=0$ or $V^\epsilon (a)=0$.
	By the same lemma we get $V=0$, a contradiction. 
	Hence \cref{ass:V0} is trivially satisfied. 
\end{proof}

\begin{lemma}
\label{Moufang non S}
	The only inner ideal strictly containing $\SS_+$ is $\LL$ itself.
\end{lemma}
\begin{proof}
	Let $I$ be a proper inner ideal containing $\SS_+$ properly.
	By \cref{Containing S} $a_+\in I$.
	Thus $I$ also contains $[a_+,[a_+,b_-]]=U_a(b)_+$.
	Since $0\neq a$ is conjugate invertible $U_a$ is invertible, we get $\A_+\leq I$.
	Hence $\psi(\A,\A)=[\A_+,\A_+]\leq [I,I]=0$, we get a contradiction by \cref{psi non-degenerate}. 
\end{proof}

\begin{proposition}
\label{Moufang skews inner}
	If $I$ is an inner ideal of $\LL$ distinct from $\mathcal S_+$, then $I=e_+(a,s)(\SS_-)$, for unique $a\in\mathcal A$, $s\in\mathcal S$.
\end{proposition}
\begin{proof}
	Consider $0\neq a\in\A$ and $0\neq s\in\SS$ arbitrary.
	Since $\A$ is division, \cite[Lemma 3.3.4(ii)]{Boelaert2019} shows that $\psi(a,U_a(sa))\neq 0$.
	By \cref{Reducing to S_+,Set Assumption,Moufang non S} it suffices to show $E(\A)(\SS_+)=\{\SS_+\}\cup E_+(\A)(\SS_-)=:S$.
	Since $E_+(\A)(S)=S$ it suffices to show $E_-(\A)(S)=S$.
	
	Consider $a\in\A$ and $s\in\SS$ arbitrary with $(a,s)\neq (0,0)$.
	We show that $e_-(a,s)(\SS_+)\in E_+(\A)(\SS_-)$.
	By \cref{-2 -1 0 1 2} it suffices to show that $e_-(a,s)(\SS_+)$ has an element with non-zero $(-2)$-component.
	Recall \cref{Image e_+} for a precise description of the elements inside $e_-(a,s)(\SS_+)$.
	Assume first $s=0$, then $e_-(a,0)(t_+)$ has $(-2)$-component $\frac{1}{24}\psi (a,U_a(ta))$ which is non-zero as soon as $t$ is non-zero, as noted before.
	Assume now $a=0$, then $e_-(0,s)(\hat s_+)$ has non-zero $(-2)$-component $s$.
	So we may assume $a\neq 0$ and $s\neq 0$.
	Then $e_-(a,s)(\hat s_+)$ has $(-2)$-component $s+\frac{1}{24}\psi (a,U_a(\hat s a))$.
	If this is non-zero we are done, hence assume it is zero.
	Note that the $(-1)$-component equals $b:=a+\frac 1 6 U_a(\hat s a)$.
	If this component is also $0$, we would get $\frac{1}{24}\psi (a,U_a(\hat s a))=\frac{1}{4}\psi (a,-a)=0$ and hence $s=0$, a contradiction. 
	Now note that the inner ideal $e_-(a,s)(\SS_+)$ contains $[e_-(a,s)(\hat s_+),[e_-(a,s)(\hat s_+),V_{t b,b}]]$, for $0\neq t\in\SS$.
	This element has $(-2)$-component $\psi (b,U_b(tb))$, which is non-zero.
	So clearly $E_-(\A)(\SS_+)\leq S$.
	This argument also shows that for any $a\in\A$ and $s\in\SS$ with $(a,s)\neq (0,0)$ we have $e_+(a,s)(\SS_-)\in E_-(\A)(\SS_+)$.
	Together with $e_-(\A)(\SS_-)=\SS_-$ this shows $E_-(\A)(S)=S$.
	
	The uniqueness claim follows from the fact that $ta=tb$ implies $a=b$  and $L_sL_t=L_{s'}L_t$ implies $s=s'$ for non-zero $t\in\SS$, since $t$ is conjugate invertible.
\end{proof}

\begin{theorem}
\label{non-ab Moufang}
	Let $\A$ be a structurable division algebra with $\mathcal S\neq 0$. Then the set of all proper non-trivial inner ideals of $\LL=K(\A)$ forms a Moufang set, with root groups 
	\begin{align*} 
		U_{\SS_+}&=E_+(\A);\\
		 U_{e_+(a,s)(\SS_-)}&=E_-(\A)^{e_+(a,s)}, \quad \text{for all } a\in\A, s\in\SS.
	\end{align*}
	This Moufang set is isomorphic to the Moufang set $\mathbb{M}(\A)$ as defined in \cite[Theorem~5.1.6]{Boelaert2019}.
\end{theorem}
\begin{proof}
	The proof of this claim is a \emph{mutatis mutandis} copy of the proof of \cref{Set Inner are Moufang}, replacing $J$ with $\mathcal A$, $e_+(j)$ with $e_+(a,s)$ etc., using \cref{Moufang skews inner} in place of \cref{Set Inner ideals}.
\end{proof}

\section{Moufang triangles}
\label{sec 6}

We now proceed to the next case, which will give rise to Moufang triangles, i.e., to Moufang projective planes.
In fact, this case will be somewhat peculiar: the geometry we will obtain, will be the \textit{dual double} of a projective plane (which is, in fact, a \textit{thin} generalized hexagon).
Likewise, the structurable algebras we will use will be the double of the division algebra coordinatizing the projective plane. (Notice, however, that they are equipped with an involution exchanging the two components; they are still central simple, as we will show below.)

\begin{remark}\label{rem:dualdouble}
    Notice that it is impossible to find the Moufang triangles directly as the geometry of the poset of all inner ideals of the TKK Lie algebra of any central simple structurable algebra.
    Indeed, the inner ideals $\SS_+$ and $\SS_-$ would either both correspond to points or both to lines.
    In the former case, there cannot be a proper inner ideal containing both $\SS_+$ and $\SS_-$ (since such an inner ideal would have to be abelian).
    In the latter case, the intersection of $\SS_+$ and $\SS_-$ is trivial and hence there would be no point incident with both lines.
    (A similar argument applies if the structurable algebra is a Jordan algebra.)
\end{remark}

\begin{construction}
\label{construction exchange algebra}
	Let $F$ be an alternative division algebra over a field of characteristic different from $2$ and $3$.
	Set $k=Z(F)$.
	Then $\A:=F\oplus F$ is a $k$-algebra, with multiplication
	\[ (a,b).(c,d)=(ac,db),\]
	for all $a,b,c,d\in F$.
	This algebra has the involution 
	\[ (x,y)\mapsto (y,x).\]
	In particular, the subspace of skew elements is
	\begin{equation}\label{triangle S}
	   \SS = \{ (x, -x) \mid x \in F \} .
	\end{equation}
	Set $\LL=K(\A)$.
\end{construction}

\begin{lemma}[{\cite[(9.15), (9.22)]{Tits2002}}, {\cite[(3.4)]{Schafer1966}}]
\label{alternative identities}
For any alternative algebra $F$ we have: 
\begin{enumerate}[\rm (i)]
	\item $[e_{\sigma (1)},e_{\sigma (2)},e_{\sigma (3)}]=\sign (\sigma)[e_1,e_2,e_3]$, for all $e_1, e_2, e_3\in F$ and $\sigma\in \Sym(3)$.
	\item\label{alt-id-inv} $e^{-1}(ef)=f=(fe)e^{-1}$, for all $e \in F^\times$ and all $f\in F$.
	\item $(fef)g=f(e(fg))$, for all $e,f,g\in F$.
\end{enumerate}
\end{lemma}

\begin{lemma}
\label{triangle central simple structurable}
	$\A$ is a central simple structurable algebra over $k$.
\end{lemma}
\begin{proof}
    In principle, this follows by showing that $\A$ is isomorphic to one of the known central simple structurable algebras (distinguishing between whether $F$ is associative or not and relying on the classification of alternative division algebras), but we believe that the following short and direct proof is instructive. However, see \cref{F+F isom} below.

	By \cite[Theorem 13]{Allison1978} and \cref{alternative identities} it suffices to show $D_{a^2,a}(\HH)=0$ for all $a\in\HH$.
	Since $F$ is power-associative, $[a,a^2]=0$.
	Hence $D_{(e,e),(e,e)^2}(f,f)=2[(e,e),(e^2,e^2),(f,f)]$, for all $e,f\in F$.
	Now $e(e^2f)=e(e(ef))=(ee^2)f$ by \cref{alternative identities}.
	Hence $D_{a^2,a}(\HH)=0$ and $\A$ is a structurable algebra.
	
	Since $k=Z(F)$, we get $Z(\A)\leq k\oplus k$.
	The fact that $Z(\A)$ is contained in $\HH$ shows $Z(\A)=k1$.
	Let $I$ be an ideal of $\A$.
	Consider $0\neq (e,f)\in I$ arbitrary. 
	If $e\neq 0$, then $e$ is invertible and hence $(g,0)=(ge^{-1},0)(e,f)\in I$, for all $g\in F$, by \cref{alternative identities}.
	Since $I$ is closed under the involution, we get $I=\A$. 
	Similarly if $f\neq 0$.
\end{proof}

\begin{remark}\label{F+F isom}
	If $F$ is a \emph{composition} (division) algebra, then $\A$ is isomorphic to $F\otimes (k\oplus k)$, where $k\oplus k$ is the split $2$-dimensional composition algebra with involution $(a,b)\mapsto (b,a)$.
	Indeed, mapping $(e,f)\in \A$  to the element $e\otimes (1,0)+\overline f\otimes (0,1)$ of $F\otimes (k\oplus k)$ yields an isomorphism of structurable algebras.
	
	If $F$ is an \emph{associative} (division) algebra, then $\A$ is a central simple associative algebra with involution, and hence belongs to one of the (classical) classes of structurable algebras. 
\end{remark}

\begin{lemma}
\label{triangle invertiblity}
	All elements $(e,f)\in\A$ with $e,f\neq 0$ are conjugate invertible, with $\widehat{(e,f)} = (f^{-1},e^{-1})$.
\end{lemma}
\begin{proof}
	By \cref{alternative identities}\cref{alt-id-inv}, we have $V_{(e,f),(f^{-1},e^{-1})}=\id$.
\end{proof}

\begin{lemma}
\label{triangle U operator}
	Let $0 \neq f\in F$. Then $U_{(f,0)}(\A)=(F,0)$ and $U_{(0,f)}(\A)=(0,F)$.
\end{lemma}
\begin{proof}
	Clearly, $U_{(f,0)}(\A)\leq (F,0)$. Conversely, let $g\in F$ be arbitrary. Using \cref{alternative identities}, we have
	\begin{align*}
	   U_{(f,0)}(0,f^{-1}gf^{-1}) &= 2((f,0)(f^{-1}gf^{-1},0))(f,0)-((f,0)(0,f))(0,f^{-1}gf^{-1}) \\
	       &= 2(g,0).
	\end{align*}
	The proof of the other statement is similar (or follows by applying the involution).
\end{proof}

\begin{lemma}
\label{triangle derivation}
	Let $D\in\Der(\A)$. Then there is some derivation $D'$ of $F$ such that $D(e,f)=(D'(e),D'(f))$ for all $e,f\in F$.
	In particular, $D(\lambda,\mu)=0$ for all $\lambda,\mu \in k$.
\end{lemma}
\begin{proof}
	Clearly $D(1,1)=0$. Since $(1,1) = (1,-1)^2$, we have
	\[ 0 = D(1,1) = D(1,-1)(1,-1) + (1,-1)D(1,-1) = 2(1,-1)D(1,-1) . \]
	Left multiplying by $(1,-1)$ implies $D(1,-1)=0$.
	Since $D$ is $k$-linear, this already implies that $D(\lambda,\mu)=0$ for all $\lambda,\mu \in k$.
	In particular, $D(1,0) = D(0,1) = 0$.
	
	For each $e \in F$, we now have $D(e,0)=D((1,0)(e,0))=(1,0)D(e,0)$.
	Hence there is some derivation $D' \in \Der(F)$ such that $D(e,0)=(D'(e),0)$ for all $e\in F$.
	Similarly, there is some derivation $D'' \in \Der(F)$ such that $D(0,f)=(0,D''(f))$ for all $f\in F$.
	Hence $D(e,f) = (D'(e), D''(f))$ for all $e,f \in F$.
	Expressing that $D(\overline{x}) = \overline{D(x)}$ for all $x \in \A$ now implies $D' = D''$.
\end{proof}

\begin{theorem}
\label{triangle Assumption (i)}
	$\A$ satisfies \cref{ass:V0}.
\end{theorem}
\begin{proof}
	By \cref{triangle central simple structurable}, $\A$ is central simple.
	By \cref{triangle S,triangle invertiblity}, \cref{ass:skew} is satisfied.
	
	Consider an arbitrary non-zero $V\in\Inst (\A)$ with $V^\delta(\SS)=0$ and $V^2=0$.
	By \cref{S1 Decom Inst}, we have $V=T_x+D$, for some $D\in \Der(\A)$, with $x=(e,f)$, for some $e,f\in F$.
	By \cref{triangle derivation}, there is a derivation $D'$ of $F$ such that $D(g,h)=(D'(g),D'(h))$ for all $g,h\in F$.
	Note that $D^\delta=D+R_{\overline{D(1)}}=D$.
	By assumption,
	\[ 0=V^\delta(1,-1)=-\psi(x,(1,-1))+D(1,-1)=(e+f,-e-f) ; \]
	using \cref{formula:V delta}. Hence $f=-e$ so $x = (e, -e) \in \SS$.
	Then
	\[ 0=V^\delta(e,-e)=-\psi((e,-e),(e,-e))+D(e,-e)=(D'(e),-D'(e)) \]
	and thus $D'(e)=0$.
	
	Now $V(1,0)=T_x(1,0)=(3e,0)$.
	Since $D'(e)=0$ and $V^2 = 0$, this implies
	\[ 0 = V^2(1,0) = T_x(3e,0) + D(3e,0) = 3 T_x(e,0) = (9e^2,0) . \]
	Hence $e^2=0$ and $e=0$, since $F$ is division.
	We conclude that $V=D$ is a derivation.
	In particular, $V^\epsilon = V$.
	
	Since $V\neq 0$, there exists $f\in F$ such that $D'(f)\neq 0$.
	Set $y = (f,f)$ and $a = V(y) = (D'(f),D'(f))$.
	Then $V(a) = 0$ and hence
	\[ U:=[V,[V,V_{y,y}]]=2V_{a,a}.\]
	Since $\overline a=a$, we get $V_{a,a} = L_{a^2}$ by \cref{V operator symmetry}. Since $\A$ is power-associative, this implies
	$V_{a,a}^2(1) = a^4$. 
	If $U^2=0$ we get $4a^4=0$ and thus $a=0$, a contradiction.
\end{proof}

\begin{lemma}
\label{triangle containing S}
	The only proper inner ideals of $\LL$ properly containing $\SS_+$ are the two inner ideals $(F,0)_+ \oplus \SS_+$ and $(0,F)_+ \oplus \SS_+$.
\end{lemma}
\begin{proof}
	Let $I$ be a proper non-trivial inner ideal properly containing $\SS_+$.
	By \cref{Containing S,triangle Assumption (i)}, $I\leq \A_+ \oplus \SS_+$, so $I$ contains some non-zero $a_+\in\A_+$.
	If $a$ were conjugate invertible, then the operator $U_a$ would be invertible, but then the fact that $-U_{a}(b)_+=[a_+,[a_+,b_-]]\in I$ for all $b\in \A$ would imply that $\A_+\leq I$.
	Since $I$ is abelian, this is in contradiction with \cref{psi non-degenerate}.
	Hence $a$ is not conjugate invertible, so by \cref{triangle invertiblity}, $a$ is contained in $(F,0)$ or $(0,F)$.
	Assume without loss of generality that $a \in (F,0)$.
	By \cref{triangle U operator}, we get $(F,0)_+ \leq I$.
	If moreover $(0,f)_+\in I$ for some some non-zero $f\in F$, then $[(0,f)_+,(f,0)_+]=\psi ((0,f),(f,0))_+=(-f^2,f^2)_+\neq 0$ yields a contradiction with the fact that $I$ is abelian.
	Hence $I=(F,0)_+\oplus \SS_+$.
	
	It only remains to show that $J:=(F,0)_+\oplus \SS_+$ is an inner ideal, i.e., that
	\[ [(F,0)_+ \oplus \SS_+, [(F,0)_+ \oplus \SS_+, \LL]] \leq (F,0)_+ \oplus \SS_+ = (F,0)_+ \oplus \LL_2 . \]
	Using the $5$-grading of $\LL$, this means that we only have to verify the following inclusions:
	\begin{align}
	   [(F,0)_+, [(F,0)_+, \A_-]] &\leq (F,0)_+, \label{F1} \\
	   [(F,0)_+, [(F,0)_+, \SS_-]] &= 0, \label{F2} \\
	   [(F,0)_+, [\SS_+, \SS_-]] &\leq (F,0)_+, \label{F3} \\
	   [\SS_+, [(F,0)_+, \SS_-]] &\leq (F,0)_+. \label{F4}
	\end{align}
	Let $e,f,x,y \in F$ be arbitrary and let $s = (x, -x)$ and $t = (y, -y)$.
	To show \cref{F1}, we compute that
	\[ [(e,0)_+ , [(f,0)_+ , (x,y)_-]] = -V_{(f,0), (x,y)}(e,0)_+ = (-fy.e - ey.f, 0)_+ \subseteq (F,0)_+ . \]
	To show \cref{F2}, we compute that
	\[ [(e,0)_+ , [(f,0)_+ , s_-]] = -V_{(e,0),(xf,0)} = 0 . \]
	To show \cref{F3}, we compute that
	\[ [(e,0)_+ , [s_+ , t_-]] = [(e,0)_+ , L_s L_t] = (-x.ye, 0)_+ \in (F,0)_+ . \]
	Finally, \cref{F4} follows from \cref{F3} by the Jacobi identity because $[\SS_+, (F,0)_+] = 0$.
\end{proof}

We are now ready to construct a point-line geometry from the proper non-trivial inner ideals of $K(\A)$.

\begin{definition}\label{MF}
	We define the point-line geometry $\Gamma = (\M, \F)$ with as points
	\[ \M := \M(\LL) := \{I\mid I\text{ is a minimal proper non-trivial inner ideal of } \LL\} \]
	and as lines
	\[ \F := \F(\LL) := \{I\mid I\text{ is a non-minimal proper non-trivial inner ideal of } \LL\} , \]
	with inclusion as incidence.
\end{definition}

\begin{lemma}
\label{triangle line is full}
	Let $J\in\F$.
	Then $J$ is the union of the elements of $\M$ it contains.
	Moreover, any two distinct such elements intersect trivially and span~$J$.
\end{lemma}
\begin{proof}
	By \cref{triangle Assumption (i),Reducing to S_+,triangle containing S}, we may assume that $J = (F,0)_+ \oplus \SS_+$.
	Since $\SS_+\in \M$, we get $e_-((f,0),0)(\SS_+)\in \M$, for all $f\in F$. 
	Note that for each $(g,-g)\in\SS$, we have
	\[ e_-((f,0),0)((g,-g)_+)=(g,-g)_+-(gf,0)_+ , \]
	since $V_{(gf,0),(f,0)}=0$.
	Moreover, by \cref{F1,F2}, also $(F,0)_+$ is an inner ideal contained in $J$.
	Since each element of $J \setminus (F,0)_+$ can be written as $(g,-g)_+-(gf,0)_+$ for some $f,g \in F$, the first claim is proved.
	
	The second claim is obvious since the intersection of inner ideals is again an inner ideal, the elements of $\M$ are minimal inner ideals and the dimension of $J$ is twice the dimension of an element of $\M$.
\end{proof}

\begin{lemma}
\label{triangle distance 1}
	Two distinct points $I,J\in\M$ are collinear if and only if $[I,J]=0$.
\end{lemma}
\begin{proof}
	If $I$ and $J$ are collinear, then they are contained in a proper inner ideal.
	Because this inner ideal is abelian, we get $[I,J]=0$.
	
	Conversely, let $I,J\in\M$ with $[I,J]=0$.
	By \cref{Reducing to S_+}, we may assume $I=\SS_+$.
	By $[I,J]=0$ and the fact that all non-zero elements of $\SS$ are conjugate invertible, we get $J\leq \LL_0\oplus \LL_1\oplus \LL_2$.
	Suppose that there is an $x := V + a_+ + s_+ \in J$ with $s\in\SS$, $a\in\A$, $V\in\Inst (\A)$ and $V\neq 0$.
	Since $[I,J]=0$, we have $V^\delta(\SS)=0$.
	If $V^2\neq 0$, then there exists some $b\in\A$ such that $[x,[x,b_-]]\in J$ has a non-zero $(-1)$-component, a contradiction.
	If $V^2=0$, then by \cref{triangle Assumption (i)}, we can find $0\neq U\in\Inst(\A)$ such that $W:=[V,[V,U]]$ satisfies $W^2\neq 0$.
	Hence $y := [x,[x,U]] \in J$ has $0$-component $W$, so we can repeat the previous argument with $x$ replaced by~$y$ and $V$ replaced by~$W$ to get a contradiction.
	We conclude that $J\leq \A_+\oplus \SS_+$.
	
	Since $J\neq I$, there exists a non-zero $a\in\A$ and an $s\in\SS$ such that $y := a_+ + s_+\in J$.
	For each $t\in\SS$, the inner ideal $J$ contains the element $[y,[y,t_-]]$, which has $0$\dash component $-V_{a,ta}$.
	Since $J \leq \LL_1 \oplus \LL_2$, we must have $V_{a,ta}=0$ for all $t \in \SS$ and hence, by \cref{a invertible V_{a,sa}}, we get $a\in (F,0)$ or $a\in (0,F)$.
	Hence $J$ intersects one of the inner ideals $(F,0)_+\oplus\SS_+$ and $(0,F)_+\oplus\SS_+$ non-trivially.
	Since $J$ is a minimal inner ideal, we conclude that either $J \subseteq (F,0)_+\oplus \SS_+$ or $J \subseteq (0,F)_+\oplus\SS_+$.
	Since these inner ideals both contain $I = \SS_+$, the points $I$ and $J$ are indeed collinear.
\end{proof}

\begin{lemma}
\label{triangle distance 3}
	Two points $I,J\in\M$ that satisfy $[I,[I,J]]\neq 0$ are at distance~$3$ in~$\Gamma$.
	Moreover, we then have $[i,j] \neq 0$ for all non-zero $i \in I, j \in J$.
\end{lemma}
\begin{proof}
	By \cref{Reducing to S_+}, we may assume that $I=\SS_+$.
	The condition $[I,[I,J]]\neq 0$ then implies that $J$ contains an element with non-zero $(-2)$-component.
	By \cref{-2 -1 0 1 2} and the fact that $\SS_+$ is fixed by $E_+(\A)$, we may assume that $J=\SS_-$.
	Since all non-zero elements of $\SS$ are conjugate invertible, this already shows that $[i,j] \neq 0$ for all $0 \neq i \in I$ and $0 \neq j \in J$.
	
	By \cref{triangle distance 1}, the points $(F,0)_+$ and $(F,0)_-$ are collinear, so we can consider the path $(\SS_+,(F,0)_+,(F,0)_-,\SS-)$ to see that the points $I$ and $J$ are at distance at most $3$.
	By \cref{triangle distance 1}, the points $I$ and $J$ are not collinear.

	Now suppose that $I$ and $J$ are at distance $2$ and let $K\in\M$ be a point collinear with both $I$ and $J$.
	By \cref{Containing S} we get $K\leq (\SS_-\oplus \A_-)\cap (\A_+\oplus \SS_+)=0$, a contradiction.
\end{proof}

The following lemma is a stronger version of \cref{triangle distance 1}.
\begin{lemma}
\label{triangle distance 1 looser}
	Two distinct points $I,J\in\M$ are collinear if and only if there is some $0 \neq j\in J$ such that $[I,j]=0$.
\end{lemma}
\begin{proof}
	Let $0 \neq j \in J$ be such that $[I,j]=0$; we have to show that $[I,J] = 0$.
	By \cref{triangle distance 3}, we already know that $[I,[I,J]] = 0$.
	
	By \cref{Reducing to S_+}, we may assume that $I=\SS_+$.
	Then $[I,[I,J]]=0$ yields $J \leq \LL_{-1} \oplus \LL_0 \oplus \LL_1 \oplus \LL_2$.
	Hence $j = b_- + V + a_+ + s_+$ for some $a,b\in\A$, $V\in\Inst (\A)$ and $s\in\SS$.
	Since all non-zero elements of $\SS$ are conjugate invertible, $[I,j]=0$ implies $b=0$ and $V^\delta (\SS)=0$.
	
	We claim that $J$ contains an element in $\LL_1 \oplus \LL_2$ with non-zero $1$-component.
	For each $c\in \A$, $J$ contains $[j,[j,c_+]]\in \LL_1 \oplus \LL_2$, which has $1$-component $V^2(c)_+$,
	so if $V^2\neq 0$, then the claim holds.
	If $V^2=0$, we use \cref{triangle Assumption (i)} to find $W\in\Inst (\A)$ such that $j' := [j,[j,W]]\in \LL_0 \oplus \LL_1 \oplus \LL_2$ has $0$-component $U$ with $U^2\neq 0$.
	Replacing $j$ by $j'$ then shows the claim.
	
	Denote this element of $J$ by $x = c_+ + t_+$ where $0 \neq c \in \A$ and $t \in \SS$.
	Then for each $d \in \A$, $J$ also contains $[x, [x, d_-]]$, which has $1$-component $-U_c(d)_+$.
	If $c$ were conjugate invertible, then $U_c$ would be an invertible operator, hence the projection of $J$ onto $\A_+$ would be all of $\A_+$.
	This contradicts the fact that $\dim J = \dim \SS < \dim \A$.
	Hence $c$ is not conjugate invertible, so $c \in (F,0)$ or $c \in (0,F)$.
	It follows that the minimal inner ideal $J$ intersects one of the two non-minimal inner ideals $(F,0)_+\oplus\SS_+$ or $(0,F)_+\oplus\SS_+$ non-trivially and is hence contained in it.
	We conclude that indeed $[I,J]=0$.
\end{proof}

\begin{lemma}
\label{triangle distance 2}
	Let $I,J\in\M$ be two distinct points such that $[J,I]\neq 0$ and $[J,[J,I]]=0$.
	Then there is exactly one point $K \in \M$ collinear with both $I$ and~$J$.
\end{lemma}
\begin{proof}
	By \cref{Reducing to S_+}, we may assume $I=\SS_+$.
	Since $[J,I]\neq 0$ we can find $j\in J$ and $i\in I$ such that $e:=[j,i]\neq 0$.
	By $[J,[J,I]]=0$ we get $J \leq \LL_{-1} \oplus \LL_0 \oplus \LL_1 \oplus \LL_2$.
	Hence $e \in \LL_1 \oplus \LL_2$ and $j = b_- + V + a_+ + s_+$ for some $a,b\in\A$, $s\in\SS$ and $V\in\Inst (\A)$.
	Note that for each $c \in \A$, $J$ contains $[j,[j,c_+]]$, which has $(-1)$\dash component $-U_b(c)_-$.
	Since $\dim J= \dim \SS < \dim \A$, \cref{triangle invertiblity} implies that $b$ is contained in $(F,0)$ or $(0,F)$.
	Hence $e$ is contained in either $(F,0)_+ \oplus \SS_+$ or $(0,F)_+ \oplus \SS_+$.
	By \cref{triangle line is full}, $e$ is contained in a unique element $E\in\M$.
	Since $I=\SS_+$, we get $[E,I]=0$, so $E$ and $I$ are collinear by \cref{triangle distance 1}.
	Since $[J,[J,I]]=0$, we get $[J,e]=0$, so \cref{triangle distance 1 looser} implies that $J$ and $E$ are collinear.
	Since $[I,J] \neq 0$, it already follows from \cref{triangle distance 1} that $I$ and $J$ are at distance $2$ in the collinearity graph of $\Gamma$.
	
	Suppose now that $K \in \M$ is another point collinear with both $I$ and $J$.
	By \cref{triangle containing S}, $E$ and $K$ are contained in either $(F,0)_+ \oplus \SS_+$ or $(0,F)_+ \oplus \SS_+$.
	Assume first that both $E,K\leq (F,0)_+ \oplus \SS_+$.
	Then \cref{triangle line is full} implies that $E\oplus K= (F,0)_+ \oplus \SS_+$.
	Since $[E,J]=0=[K,J]$, this implies that also $[I,J]=[\SS_+,J]=0$, a contradiction.
	
	Hence we may assume that $E\leq (F,0)_+ \oplus \SS_+$ and $K\leq (0,F)_+ \oplus \SS_+$.
	Since $E\cap \SS_+=0$ by \cref{triangle line is full}, the projection of $E$ onto $\A_+$ is $(F,0)_+$.
	Similarly, the projection of $K$ onto $\A_+$ is $(0,F)_+$.
	Hence the projection of $E \oplus K \leq \A_+ \oplus \SS_+$ onto $\A_+$ is all of $\A_+$.
	Since $[E \oplus K,J]=0$, this implies that for each $c \in \A$, we can find some $t \in \SS$ such that $[c_+ + t_+, J] = 0$.
	Now let $j = b_- + V + a_+ + s_+ \in J$ be arbitrary.
	Then $[c_+ + t_+, j]$ has $0$-component $V_{c,b}$, hence $V_{c,b}=0$ for all $c\in\A$ and in particular $U_c(b) = V_{c,b}(c) = 0$ for all $c \in \A$. By choosing $c$ conjugate invertible, we get $b=0$.
	This implies that $[c_+ + t_+, j]$ has $1$-component $V(c)$, hence $V(c) = 0$ for all $c \in \A$, so also $V = 0$.
	So $J \leq \A_+ \oplus \SS_+$, but then $[I,J]=0$, which is again a contradiction.
\end{proof}

\begin{corollary}
\label{triangle distances}
	Let $I,J\in\M$ be two distinct points at distance $d$ in the collinearity graph of $\Gamma$. Then
	\begin{align*}
	   d = 1 &\iff [I,J]=0, \\
	   d = 2 &\iff [I,J]\neq 0 \text{ and } [I,[I,J]]=0, \\
	   d = 3 &\iff [I,[I,J]]\neq 0.
	\end{align*}
\end{corollary}

\begin{lemma}
\label{triangle no n-gons}
	The geometry $\Gamma$ does not contain triangles, quadrangles or pentagons.
\end{lemma}
\begin{proof}
    Let $I \in \M$ be a point and suppose that $I$ would be contained in a triangle $(I,J,K)$ in $\Gamma$.
    By \cref{triangle Assumption (i),Reducing to S_+}, we may assume that $I = \SS_+$.
	By \cref{triangle containing S}, there are only two lines through $I$, namely $(F,0)_+ \oplus \SS_+$ and $(0,F)_+ \oplus \SS_+$,
	so we must have $J \subseteq (F,0)_+ \oplus \SS_+$ and $K \subseteq (0,F)_+ \oplus \SS_+$ (or conversely).
	Since $J \neq I \neq K$, there is some $(e,0)_+ + s_+ \in J$ and $(0,f)_+ + t_+ \in K$ with $e,f \neq 0$.
	Since $J$ and $K$ are collinear, we have $[J,K] = 0$ by \cref{triangle distance 1}, so in particular
	$0 = [(e,0)_+ + s_+, (0,f)_+ + t_+] = \psi((e,0),(0,f))_+$, which is a contradiction.
	Hence $\Gamma$ does not contain triangles.
	
	By \cref{triangle distance 2,triangle distances}, $\Gamma$ does not contain quadrangles.
	
	Suppose now that $\Gamma$ would contain a pentagon $(I,J,M,N,O)$ for certain points $I,J,M,N,O \in \M$.
	By \cref{triangle Assumption (i)}, we may again assume that $I=\SS_+$.
    By \cref{triangle containing S}, we may then assume that $J$ contains some $j = a_+ + s_+$ with $a=(f,0)$ for some non-zero $f\in F$.
    Write $s=(g,-g)$ with $g\in F$.
    Then
    \[ e_+((0,f^{-1}g),0)(a_++s_+)=a_+ . \]
    Since $e_+((0,f^{-1}g),0)$ fixes $I=\SS_+$, we may assume that $J=(F,0)_+$ as well.
    Notice that $\langle (F,0)_+,(F,0)_-\rangle$ is a line through $I$ and that any point lies on exactly two lines by \cref{triangle containing S}; hence $M\leq \langle (F,0)_+,(F,0)_-\rangle$.
    So $M$ contains some $y:=(f,0)_++(g,0)_-$ with $f, g\in F$.
    Note that $g \neq 0$ since $J \cap M = 0$.
    We now observe that
    \[ e_+ \bigl( 0,(-fg^{-1},fg^{-1}) \bigr)(y) = (g,0)_- \]
    and that $e_+(0,(-fg^{-1},fg^{-1}))$ preserves $I$ and $J$. We may thus assume that $M=(F,0)_-$.
    
    Since $O$ is collinear with $I$ and not on the line $IJ$, \cref{triangle containing S} shows that $O$ contains some element $s_++(0,f)_+$ with $s\in\SS$ and $0 \neq f \in F$.
    Since $O$ and $M$ are at distance at most $2$, we get $[O,[O,M]]=0$.
    Hence
    \[ [s_++(0,f)_+,[s_++(0,f)_+,(g,0)_-]]=0 \]
    for all $g\in F$.
    In particular, the $1$-component $-U_{(0,f)}(g,0)$ must be $0$ for all $g\in F$.
    Since $U_{(0,f)}(g,0)=2(0,fgf)$, we get a contradiction.
    We conclude that $\Gamma$ does not contain pentagons.
\end{proof}

\begin{theorem}
\label{triangle root groups}
	Let $\A$ be the structurable algebra from \cref{construction exchange algebra} and set $\LL=K(\A)$. 
	Consider the graph $\Omega=(V,E)$, with
	\begin{align*}
	   V &= \{ I \mid I \text{ is a proper non-minimal non-trivial inner ideal of $\LL$} \} , \\
	   E &= \{ \{ I,J \} \mid I \cap J\neq 0 \} . 
	\end{align*}
	Then $\Omega$ is the incidence graph of a Moufang triangle.
	Its root groups can be identified with
	\[ U_1=e_-((F,0),0), U_2=e_-(0,\SS), U_3=e_-((0,F),0),  \]
	and the commutator relations are as in $E_-(\A)$.
	
	Moreover, the geometry $\Gamma$ (as introduced in \cref{MF}) is a thin generalized hexagon, which is the dual double%
	\footnote{The \textit{dual double} of a point-line geometry $\Delta$ is the geometry with as point set the \textit{flags} of $\Delta$, i.e., the incident point-line pairs, and as line set the union of the point set and the line set of $\Delta$.}
	of this Moufang triangle.
\end{theorem}
\begin{proof}
	By \cref{triangle distance 1}, the cycle of points
	\[ \bigl( \SS_+, \, (F,0)_+, \, (F,0)_-, \, \SS_-, \, (0,F)_-, \, (0,F)_+, \, \SS_+ \bigr) \]
	forms a hexagon in $\Gamma$.
	Then \cref{triangle Assumption (i),triangle containing S,triangle distances,triangle no n-gons} show that $\Gamma$ is a thin generalized hexagon.
	
	Note that the intersection of two non-minimal proper inner ideals is a minimal inner ideal if this intersection is non-empty. This intersection is a point of the thin generalized hexagon $\Gamma$.
	Therefore, $\Omega$ is the incidence graph of a generalized triangle and $\Gamma$ is the dual double of this triangle.

	Consider the following $6$-cycle in $\Omega$:
	\begin{multline*} 
		(x_0,\dots,x_5,x_0) = \bigl( (0,F)_+\oplus \SS_+, \ \SS_+ \oplus (F,0)_+, \ (F,0)_+\oplus (F,0)_-, \ (F,0)_-\oplus \SS_-,\\
		  \SS_-\oplus (0,F)_-, \ (0,F)_-\oplus (0,F)_+, \ (0,F)_+\oplus \SS_+ \bigr) .
	\end{multline*}
	We now show that $e_-((F,0),0)$ fixes all neighbors of $x_2$ and $x_3$ and acts transitively on the set of all neighbors of $x_1$ distinct from $x_2$.
	Note that if $I$ is a proper non-minimal inner ideal, we can identify a neighbor $J$ of $I$ with the minimal inner ideal $I\cap J$, since $\Gamma$ is a thin generalized hexagon.
	Moreover, any automorphism fixing $I$ and $I\cap J$ fixes $J$.
	So in order to check that $e_-((F,0),0)$ fixes all neighbors of $x_i$ it suffices to check that it fixes $x_i$ itself and all the minimal inner ideals it contains, for $i=2,3$, and this is
	clear since it fixes the inner ideals $x_2$ and $x_3$ elementwise.
	Similarly, in order to check that that $e_-((F,0),0)$ acts transitively on the set of all neighbors of $x_1$ distinct from $x_2$, it suffices to show that it acts transitively on the set of all minimal inner ideals in $x_1$ distinct from ${x_1}\cap {x_2}=(F,0)_+$.
	This follows from the fact that $[(f,0)_-,(g,-g)_+]=(-fg,0)_+$ for all $f,g \in F$ and a short computation.

	By \cref{Moufang general remark}, it now follows that the root group $U_1$ coincides with $e_-((F,0),0)$;
	similarly, the root group $U_3$ coincides with $e_-((0,F),0)$.
	In order to determine the root group $U_2$, note that $e_-(0,\SS)$ fixes all neighbors of $x_3$ and $x_4$ since $[\SS_-,x_i]=0$, for $i=3,4$.
	Using the fact that $[(f,-f)_-,(g,0)_+]=(fg,0)_-$, we deduce that it acts transitively on the set of all neighbors of $x_2$ different from $x_3$. 
	
	In order to show that $\Omega$ is the incidence graph of a Moufang triangle it suffices to show that $E(\A)$ acts transitively on the set of all cycles of length $6$.
	Consider any $6$-cycle $(y_0,\dots,y_5,y_0)$.
	By \cref{Reducing to S_+}, the minimal inner ideal $y_0 \cap y_1$ can be mapped onto $\SS_+$.
	Consider the minimal inner ideal $I = y_3 \cap y_4$.
	By \cref{triangle distances}, we get $[\SS_+,[\SS_+,I]]\neq 0$.
	Then by \cref{-2 -1 0 1 2}, there exists an automorphism fixing $\SS_+$ and mapping $I$ onto $\SS_-$.
	Since there are precisely two proper inner ideals through a minimal inner ideal and $[(0,F)_+,\SS_-]\neq 0$, there exists an automorphism mapping our cycle $(y_0,\dots,y_5,y_0)$ onto the cycle $(x_0,\dots,x_5,x_0)$.
\end{proof}

\begin{remark}
	Another way to obtain the Moufang triangles associated with an \textit{octonion} division algebra $\mathbb O$, is to consider the proper non-trivial inner ideals of the exceptional Jordan algebra $H_3(\mathbb O)$ (not of its TKK Lie algebra!).
	These inner ideals are either $1$- or $10$-dimensional; the former can be identified with the points of the projective plane and the latter with the lines of the projective plane.
	This is, of course, related to the fact that these Moufang triangles arise as rank $2$ forms of buildings of type $E_6$.
	See \cite[2(B)]{Faulkner1973} and \cite[p. 34-35]{McCrimmon2004}.
\end{remark}

\begin{remark}
\label{Triangle other grading}
    As we have mentioned in the introduction, a Lie algebra can often be obtained in more than one way as the TKK Lie algebra of a structurable algebra; see \cite[Theorem 5.9]{Stavrova2017}.
    Moreover, by \cite[5.2]{Lopez2007}, \textit{any} inner ideal is the ``end'' of a $\mathbb Z$-grading, at least if $\Char k \neq 2,3,5$.
    For some choices of inner ideals, this $\mathbb Z$-grading will be different from the standard $5$-grading obtained via the TKK construction.

	In the case that we are currently considering, the inner ideal $\SS_+$ is, of course, the end of a $\mathbb Z$-grading on $\LL$.
	On the other hand, consider an inner ideal contained in $\F$ (i.e., a non-minimal inner ideal); then this inner ideal will also be the end of a $\mathbb Z$-grading on $\LL$.
	Indeed, note that $\ad_{T_{(1,-1)}}$ is a grading derivation with components
	\begin{align*} 
		\LL_{-1}&=(0,F)_-\oplus (0,F)_+ \\
		\LL_0&=\SS_-\oplus\Inst(\A)\oplus \SS_+ \\
		\LL_1&=(F,0)_-\oplus (F,0)_+.
	\end{align*}
	Then $(\LL_{-1},\LL_1)$ is a \textit{Jordan pair} (as defined in \cite[1.2]{Loos1975}).
	More precisely, one checks that $(\LL_{-1},\LL_1)$ is isomorphic to the Jordan pair $(M_{1,2}(F),M_{1,2}(F^{\operatorname{opp}} ))$ via the isomorphism 
	\begin{align*} 
		(0,f)_-+(0,g)_+&\mapsto 2 (\overline f, \overline g);\\
		(f,0)_-+(g,0)_+&\mapsto  (g,f),
	\end{align*}
	using \cite[8.15]{Loos1975}.
\end{remark}

\section{Moufang hexagons}
\label{sec 7}

We now come to the case of the Moufang hexagons.
It is known from the classification of Moufang hexagons by J.~Tits and R.~Weiss \cite{Tits2002} that Moufang hexagons are parametrized by anisotropic cubic norm structures, or equivalently, by cubic Jordan division algebras. The structurable algebra we will use will be a matrix structurable algebra constructed from this Jordan algebra.

\begin{notation}\label{notation sec 7}
	let $J$ be a cubic Jordan division algebra over a field of characteristic different from $2$ and $3$, with non-degenerate admissible form $N$, trace form $T$ and Freudenthal cross product $\times$. 
	Let $\A$ be the structurable algebra $M(J,1)$ of skew-dimension one. (Recall \cref{def:struct matrix}.)
	Let $s := \begin{psmallmatrix} 1 & 0 \\ 0 & -1 \end{psmallmatrix} \in \SS$.
\end{notation}

	Set $\LL=K(\A)$. 
	The aim of this section is to prove that the geometry of the proper non-trivial inner ideals of $\LL$, with inclusion as incidence, is a generalized hexagon.

We begin with a description of the \textit{derivations} of $\A$; this will be used in the proof of \cref{Hexagon Assumption (i)} below.
\begin{lemma}
\label{Skew dimension derivation}
	Let $D\in\Der(\A)$. Then there exist $m, n\in \End(J)$ such that
	\[  D \begin{pmatrix}
		\alpha & l \\ j & \beta\end{pmatrix} = \begin{pmatrix}
		0 & m(l) \\ n(j) & 0\end{pmatrix} \]
	for all $\alpha,\beta\in k$ and $l,j\in J$. If $D\neq 0$, then $m\neq 0\neq n$.
\end{lemma}
\begin{proof}
    The fact that $D(1)=0$ is evident. Since $D(\SS) \subseteq \SS$, we have $D(s)=\lambda s$ for some $\lambda\in k$. Hence $0=D(1)=D(s^2) = D(s) s + s D(s) = 2\lambda s^2=2\lambda$ and thus $D(s)=0$.
    This already shows that $D \begin{psmallmatrix} \alpha & 0 \\ 0 & \beta \end{psmallmatrix}$ for all $\alpha,\beta \in k$.
    
    We now consider an element of the form
    \[ x=\begin{pmatrix}
		0 & l \\ 0 & 0
	\end{pmatrix} .\]
	Then $sx=x$ and $xs=-x$. Hence $D(x) = D(s)x + sD(x) = sD(x)$ and $-D(x) = D(x)s + xD(s) = D(x)s$. This implies that
	\[ D(x) = \begin{pmatrix}
		0 & l' \\ 0 & 0
	\end{pmatrix} \]
	for some $l'\in J$. Similarly, for each $j \in J$, there is some $j' \in J$ such that $D \begin{psmallmatrix} 0 & 0 \\ j & 0 \end{psmallmatrix} = \begin{psmallmatrix} 0 & 0 \\ j' & 0 \end{psmallmatrix}$.
 	Since $D$ is $k$-linear, we conclude that there exist $m, n\in \End(J)$ such that
	\[  D \begin{pmatrix}
		\alpha & l \\ j & \beta\end{pmatrix} = \begin{pmatrix}
		0 & m(l) \\ n(j) & 0\end{pmatrix} \]
	for all $\alpha,\beta\in k$ and $l,j\in J$. 

	Assume now that $m = 0$; we will show that this implies $D = 0$. We have $x^2 = \begin{psmallmatrix} 0&0\\l\times l&0\end{psmallmatrix}$ and $l\times l=2l^\sharp$.
	Since $m = 0$, we have $D(x) = 0$.
	Hence also $D(x^2) = D(x)x+xD(x) = 0$, which implies $n(l^\sharp)=0$.
	Since $l \in J$ was arbitrary and $J$ is division, this implies that $n = 0$.
\end{proof}

\begin{theorem}
\label{Hexagon Assumption (i)}
    Let $\A$ be as in \cref{notation sec 7}. Then $\A$ satisfies \cref{ass:V0}.
\end{theorem}
\begin{proof}
	Notice that $\A$ is central simple by \cite[\S 4, Lemma 2.1]{Allison1984}. 
	\cref{ass:skew} is now obviously satisfied.
	
	Consider $V\in\Inst(\A)$ such that $V^2=0$, $V^\delta(s)=0$ and $V\neq 0$.
	By \cref{S1 Decom Inst}, we can write $V = D+T_x$ for some $D\in\Der(\A)$ and $x\in \A$. Write
	\[ x= \begin{pmatrix}
		\alpha & l \\ j & \beta
	\end{pmatrix} \]
	for some $\alpha,\beta\in k$ and $j,l\in J$.
	Let $m,n \in \End(J)$ be as in \cref{Skew dimension derivation}.
	Then $D(1)=0=D(s)$, and thus $D^\delta(s)=0$ and $D^\epsilon=D$. On the other hand, $T_x^\delta(s)=-\psi(x,s)=xs+s\overline x=\bigl ( \begin{smallmatrix}  \alpha+\beta&0\\0&-\alpha-\beta \end{smallmatrix}\bigr )$. Hence the condition $V^\delta(s)=0$ implies $\beta=-\alpha$. In particular, $x - \overline x = 2\alpha s$.
	
	Now let $y = \begin{psmallmatrix} 1&0\\0&0\end{psmallmatrix}$. Then
	\[ T_x(y) = xy+y(x-\overline x) = xy+2\alpha ys = \begin{pmatrix}
		3\alpha &0\\j&0 
	\end{pmatrix} \]
	and hence
	\begin{align*}
	T_x^2(y) &= x T_x(y) + T_x(y) (x-\overline x) = \begin{pmatrix}
		3\alpha^2+ T(l,j) +6\alpha^2 & j\times j\\ 2\alpha j+2\alpha j &0
	\end{pmatrix} \\ 
	&=\begin{pmatrix} 9\alpha^2+ T(l,j)& 2 j^\sharp \\ 4\alpha j & 0\end{pmatrix}.	
	\end{align*}
	Using \cref{Skew dimension derivation} and the fact that $D(y) = 0$, we get 
	\[ 0 = V^2(y) = D(T_x(y)) + T_x^2(y) =  \begin{pmatrix} 9\alpha^2+ T(l,j)& 2 j^\sharp \\ 4\alpha j+n(j) & 0\end{pmatrix} . \]
    In particular, $j^\sharp = 0$; since $J$ is division, this implies $j=0$. Hence $9\alpha^2=0$, i.e., $\alpha=0$. 
	By considering $\bigl (\begin{smallmatrix}0&0\\0&1\end{smallmatrix}\bigr )$ instead of $y$, we obtain in a similar fashion that $l=0$. We conclude that $x=0$.
	
	Hence $V=D$ for some $D\in \Der (\A)$.
 	Since $V \neq 0$, it follows from \cref{Skew dimension derivation} that we can find some $j\in J$ with $l:=m(j)\neq 0$.
 	Consider $a=\bigl (\begin{smallmatrix} 0&j\\ 0&0\end{smallmatrix}\bigr )$ and notice that $D(a)=\bigl (\begin{smallmatrix} 0&l\\ 0&0\end{smallmatrix}\bigr ) \in \HH$, $D(a)^2=\bigl (\begin{smallmatrix}0&0\\ 2 l^\sharp &0\end{smallmatrix}\bigr )$ and $(D(a)^2)^2 =\bigl (\begin{smallmatrix} 0&8l^{\sharp \sharp} \\ 0 &0\end{smallmatrix}\bigr )$.
	Since $V^2=0$ and $V^\epsilon = V$, we have
	\[ U:=[V,[V,V_{a,a}]] = 2V_{D(a),D(a)} = 2 L_{D(a)^2} \]
	by \cref{V operator symmetry}.
	If $U^2 = 0$, then $L_{D(a)^2}^2 = 0$, which can be applied on $1$ to get $(D(a)^2)^2 = 0$, hence $8l^{\sharp\sharp} = 8 N(l) l=0$.
	Since $J$ is division, this implies $l=0$ and we get a contradiction. 
	Hence $U^2\neq 0$.
\end{proof}

We now consider the extremal geometry associated to $\LL$.
Recall from \cref{extremal geometry} that $\E(\LL)$ is the collection of one-dimensional inner ideals of $\LL$ and that the non-zero elements of these inner ideals are precisely the extremal elements of~$\LL$. The lines are two-dimensional subspaces of $\LL$ such that every one-dimensional subspace of it is an element of $\E(\LL)$.
\begin{proposition}
\label{Characterization extremal J +}
	Let $a$ be a non-zero element of $\A$. The following are equivalent:
	\begin{enumerate}[\rm (a)]
		\item\label{ch:a} $a_+$ is extremal;
		\item\label{ch:b} $U_a(\A)\leq \langle a\rangle$; 
		\item\label{ch:c} $U_{sa}(\A)\leq \langle sa\rangle$; 
		\item\label{ch:d} $V_{a,sa}=0$;
		\item\label{ch:e} $\langle s_+,a_+\rangle$ is a line of $\Gamma(\LL)$.
	\end{enumerate}
\end{proposition}
\begin{proof}
	First note that $s_+\in\SS_+$ is extremal since $\langle s_+\rangle=\SS_+$ is the $2$-component of the $5$-grading on $\LL$. Recall that $s^2=1$.
	\begin{itemize}\itemsep1ex
	\item[\cref{ch:a}$\implies$\cref{ch:b}.]
    	If $a_+$ is extremal, then for each $b \in \A$, the element $[a_+,[a_+,b_-]]=-U_a(b)_+$ must be a multiple of $a_+$.
    	
	\item[\cref{ch:b}$\implies$\cref{ch:c}.]
    	This follows immediately from the identity $U_{sa}=-L_sU_aL_s$ (see \cite[Proposition 11.3]{Allison1981}).
    	
	\item[\cref{ch:c}$\implies$\cref{ch:d}.]
        Let $x:=e_-(-sa,0)(s_+)$.
    	Since $s_+$ is extremal, so is $x$.
    	Using \cref{Image e_+} together with the fact that $s(sa)=(s^2)a=a$, we get
    	\[ x = \tfrac{1}{24} \psi(sa, U_{sa}(a))_--\tfrac 1 6 U_{sa}(a)_-+\tfrac 1 2 V_{a,sa}+a_++s_+ . \]
    	Now $U_{sa}(a)\in\langle sa\rangle$ by assumption. Since $\psi(sa,sa)=0$, it follows that
    	\[ x=(\lambda sa)_-+\frac{1}{2}V_{a,sa}+a_++s_+ \]
    	for some $\lambda \in k$.
    	
    	Assume first that $\lambda\neq 0$. Since $[x,[x,s_+]]\leq \LL_0\oplus \LL_1\oplus \LL_2$, the fact that $x$ is extremal but has a non-zero $(-1)$-component implies that $[x,[x,s_+]]=0$. 
    	In particular, the $0$-component $\lambda^2V_{a,sa}$ equals $0$ and hence $V_{a,sa} = 0$.
    	
    	Assume next that $\lambda=0$. Let $b \in \A$ be arbitrary. Then $[x,[x,b_-]]$ must be a multiple of $\frac{1}{2}V_{a,sa}+a_++s_+$, but this element has $(-1)$-component $\tfrac{1}{4} V_{sa,a}^2(b)_-$.
    	Since $b$ was arbitrary, this implies $V_{sa,a}^2=0$.
    	Moreover, $V_{a,sa}^\delta(s)=-\psi(a,a)=0$. 
    	By \cref{V^2 equiv with (V^epsilon)^2}, we also have $V_{a,sa}^2 = 0$.
    	Suppose now that $V_{a,sa}\neq 0$.
    	By \cref{Hexagon Assumption (i)}, there exists $W\in\Inst(\A)$ such that $[V_{a,sa},[V_{a,sa},W]]\not\in \langle V_{a,sa}\rangle$. 
    	Since $x$ is extremal and $[V_{a,sa},[V_{a,sa},W]]$ is $4$ times the $0$-component of $[x,[x,W]]$, this is a contradiction. 
    	Hence $V_{a,sa}=0$ also in this case.
    	
	\item[\cref{ch:d}$\implies$\cref{ch:e}.]
        Let $\lambda \in k$ be arbitrary and let $x := e_-(-\lambda sa,0)(s_+)$.
    	Since $V_{a, sa} = 0$ by assumption, we have
    	\[ \ad(-\lambda (sa)_-)^2 (s_+)=\lambda^2 V_{a,sa}=0 \]
    	and hence $\ad(-\lambda (sa)_-)^j (s_+)=0$ for all $j\geq 2$.
    	It follows that $x = \lambda a_+ + s_+$, and since $s_+$ is extremal, we already obtain that $\lambda a_++s_+$ is also extremal for all $\lambda\in k$.
    	In particular, the element $a_++s_+$ is extremal.
    	
    	It remains to show that $a_+$ is extremal; the result will then follow because $[s_+,a_+]=0$.
    	Note that for any $i \in [-2, 2]$, we have
    	\[ [a_++s_+,[a_++s_+,\LL_i]]\leq \LL_{i+2}\oplus \LL_{i+3}\oplus \LL_{i+4} \]
    	and the projection of $[a_++s_+,[a_++s_+,\LL_i]]$ onto  $\LL_{i+2}$ equals $[a_+,[a_+,\LL_i]]$. 
    	Together with $[a_++s_+,[a_++s_+,\LL_i]]\leq \langle a_++s_+\rangle$ this implies $[a_+,[a_+,\LL_i]]\leq \langle a_+\rangle$. 
    	We conclude that $a_+$ is indeed extremal.
    	
	\item[\cref{ch:e}$\implies$\cref{ch:a}.]
    	This is obvious.
	\qedhere
	\end{itemize}
\end{proof}

\begin{definition}[{\cite[Definition 6.1]{Garibaldi2001}}]
	Let $I$ be a subspace of a structurable algebra $\A$ of skew-dimension one. Then $I$ is called an \emph{inner ideal} of $\A$ if $U_i(\A)\leq I$, for all $i\in I$.
\end{definition}

\begin{corollary}
\label{Inner ideal skew one}
	All proper non-trivial inner ideals of $\A$ are $1$-dimensional. 
\end{corollary}
\begin{proof}
	Let $I$ be a proper non-trivial inner ideal of $\A$.
	Consider $a\in I$.
	Let $x:=e_-(-sa,0)(s_+)$.
    Since $s_+$ is extremal, so is $x$.
    Using \cref{Image e_+} we get
    	\[ x = \tfrac{1}{24} \psi(sa, U_{sa}(a))_--\tfrac 1 6 U_{sa}(a)_-+\tfrac 1 2 V_{a,sa}+a_++s_+ . \]
    Now since $a$ is not conjugate invertible (otherwise $I=\A$), we get $\psi(sa, U_{sa}(a))=0$ by \cite[Theorem 2.11]{Allison1984}.
    Hence
    	\[ x=(sa')_-+\frac{1}{2}V_{a,sa}+a_++s_+ \]
    for some $a'\in I$.
	For any $b\in \A$ we get that $[x,[x,b_+]]$ has $(-1)$-component $-U_{sa'}(b)$.
	Since $x$ is extremal we get $U_{sa'}(\A)\leq \langle sa'\rangle $.
	By \cref{Characterization extremal J +} we get that $\langle a'\rangle$ is a $1$-dimensional inner ideal contained in $I$.
	Then \cite[Theorem 6.12]{Garibaldi2001} concludes this proof.
\end{proof}

\begin{remark}
	By \cref{Characterization extremal J +}, an element $a$ of the structurable algebra $\A$ is \emph{extremal}%
	\,\footnote{This notion coincides with the notion of a \textit{singular element} in \cite[Definition 5.1]{Garibaldi2001} and a \emph{strictly regular element} in \cite[p. 196]{Allison1984}.},
	i.e.\@~$U_a(\A)\leq\langle a\rangle$, if and only if $a_+$ is extremal (in the Lie algebra).
	Note that the proof uses the fact that we are working with a structurable algebra of skew-dimension one and the implication from (c) to (d) relies on the assumption that $J$ is division (via \cref{Hexagon Assumption (i)}). 
	One could expect that there is a more direct proof of the equivalence of (b) and (d), without going to the Lie algebra, which may hold in a more general setting.
\end{remark}

\begin{corollary}
\label{Extremal J_+ explicit}
	The set of extremal elements of $\LL$ contained in $\A_+$ equals 
	\[B:=\Biggl \{\lambda \begin{pmatrix} N(x) & x \\ x^\sharp & 1 \end{pmatrix} _+ \Bigm\vert x\in J, \lambda\in k^\times\Biggr\} \cup \Biggl \{ \lambda\begin{pmatrix} 1 & 0 \\ 0 & 0 \end{pmatrix} _+ \Bigm\vert \lambda\in k^\times\Biggr \} .\]
	Moreover, if $\psi(a,b)=0$ for $a,b \in B$, then $a$ and $b$ are linearly dependent.
\end{corollary}
\begin{proof}
	By \cref{Characterization extremal J +}, $a_+$ is extremal if and only if $U_a(\A)\leq \langle a\rangle$.
	Recall that $J$ is a division algebra, so in particular, if $j \in J$, then $(j^\sharp)^\sharp = N(j) j$ and if $j^\sharp=0$, then $j=0$.
	The first statement now follows from \cite[Lemma 5.7]{Garibaldi2001}.
	
	A straightforward calculation shows that for any $x,y\in J$,
	\[ \psi \left( \begin{pmatrix} N(x) & x \\ x^\sharp & 1 \end{pmatrix}_+ , \begin{pmatrix} N(y) & y \\ y^\sharp & 1 \end{pmatrix}_+ \right) = \lambda s_+ \]
	with $\lambda=N(x)-N(y)+T(x,y^\sharp)-T(y,x^\sharp)=N(x-y)$.
	Since $J$ is division, $\lambda=0$ if and only if $x=y$.
	Finally observe that
	\[ \psi \left( \begin{pmatrix} N(x) & x \\ x^\sharp & 1 \end{pmatrix}_+ , \begin{pmatrix} 1 & 0 \\ 0 & 0 \end{pmatrix}_+\right) = -s_+.\]
	The second statement is now clear.
\end{proof}

\begin{remark}
\label{Hexagon embedding Moufang set}
	By \cref{Inner ideal skew one} the only proper non-trivial
	inner ideals of the structurable algebra $\A$ are $1$-dimensional, and form the Moufang set corresponding to $J$ (see \cite{DeMedts2020}).
	\Cref{Characterization extremal J +} shows that if $x$ is an extremal element of $\A$, then $x_+$ is an extremal element of $\LL=K(\A)$. 
	Hence, we get an embedding of the Moufang set into the geometry of inner ideals of $\LL$, which will turn out to be the Moufang hexagon associated to $J$; see \cref{Hexagon Main Theorem} below.
\end{remark}

\begin{lemma}
\label{Hexagon inner containing S_+}
	Any proper inner ideal $I$ containing $\SS_+$ is either $\SS_+$ itself or equals $\langle a_+,s_+ \rangle$ for some extremal element $a_+\in\A_+$. Moreover, any non-zero element in $I$ is extremal.
\end{lemma}
\begin{proof}
	We may assume $I \neq \SS_+$. 
	By \cref{Hexagon Assumption (i)} and \cref{Containing S}, we get $\SS_+ \lneqq I\leq \A_+\oplus \SS_+$.
	Hence $a_+\in I$ for some non-zero $a\in \A$.	
	Then $I$ also contains $[a_+,[a_+,s_-]]=-V_{a,sa} \in \LL_0$, hence $V_{a,sa}=0$. By \cref{Characterization extremal J +}, $a_+$ is then extremal.
	
	Suppose that $I$ is at least $3$-dimensional. 
	By the previous paragraph, any element of $I$ is of the form $\mu a_++\lambda s_+$, for some $\lambda,\mu\in k$ and some extremal element $a_+$. 
	Since $\dim(I)\geq 3$, we can find two linearly independent extremal elements $a_+$ and $b_+$ in $I \cap \A_+$.
	Since $I$ is abelian, $\psi(a_+,b_+)=[a_+,b_+]=0$, but this contradicts \cref{Extremal J_+ explicit}. 
	Hence $I$ must be $2$-dimensional and hence equal to $\langle a_+,s_+\rangle$ for some extremal element $a_+\in\A_+$.
	The last claim follows from \cref{Characterization extremal J +}\cref{ch:e}.
\end{proof}

\begin{theorem}
\label{Hexagon Main Theorem}
	Let $\A$ be the structurable algebra $M(J,1)$ (of skew-dimension one) over a field of characteristic different from $2$ and $3$, where $J$ is a cubic Jordan division algebra. 
	Set $\LL=K(\A)$. 
	Consider the graph $\Omega=(V,E)$, with
	\begin{align*}
	   V &= \{ I \mid I \text{ is a proper non-trivial inner ideal of $\LL$} \} , \\
	   E &= \{ \{ I,J \} \mid I \lneq J \} . 
	\end{align*}
	Then $\Omega$ is the incidence graph of a generalized hexagon.
\end{theorem}
\begin{proof}
	Let $I$ be an arbitrary proper non-trivial inner ideal. 
	By \cref{Hexagon Assumption (i)} and \cref{Reducing to S_+}, there exists an element of $E(\A)$ mapping $I$ to an inner ideal containing $\SS_+$. 
	By \cref{Hexagon inner containing S_+}, the only proper non-trivial inner ideals containing $\SS_+$ are $\SS_+$ itself and the inner ideals $\langle a_+,s_+\rangle$ for an extremal element $a_+\in\A_+$.
	Moreover, all non-zero elements of such an inner ideal are extremal, i.e., these inner ideals are singular.
	This implies that any line in the extremal geometry $(\E(\LL),\F(\LL))$ is a maximal singular subspace. 
	\Cref{Extremal geometry characterisation hexagon} now implies that $(\E(\LL),\F(\LL))$ is a generalized hexagon (notice that the conditions of this theorem are satisfied by \cref{A nondeg} and \cref{Extremal J_+ explicit}). 
	Since all proper non-trivial inner ideals are either points or lines of the extremal geometry, we conclude that $\Omega$ is the incidence graph of a generalized hexagon.
\end{proof}

\begin{remark}
\label{Remark Faulkner hexagon}
	In \cite[Chapter 11]{Faulkner1977}, Faulkner defines a Lie algebra starting from a Jordan cubic division algebra $J$, which we will denote by $F(J)$. 
	There is a lot of evidence that this Lie algebra is isomorphic to $K(M(J,1))$, but we have not pursued this in detail.
	Indeed, in \cite[Chapter 12]{Faulkner1977}, Faulkner proves that the geometry with as points the $1$\dash dimensional inner ideals of $F(J)$ and as lines the $2$\dash dimensional inner ideals of $F(J)$ containing at least two $1$\dash dimensional inner ideals, with inclusion as incidence, form a generalized hexagon. 
	If it is indeed true that $K(M(J,1))\cong F(J)$, then \cref{Hexagon Main Theorem} is a generalization of Faulkners' result, in the sense that we are considering all inner ideals (rather than only the $1$\dash dimensional ones and the $2$\dash dimensional ones containing at least two $1$\dash dimensional ones). 
	Moreover, our approach also allows to identify the Moufang sets associated to $J$ in the Moufang hexagon associated to $J$ (see \cref{Hexagon embedding Moufang set}).
\end{remark}

\begin{remark}
\label{Hexagon other grading}
    Let us again (as in \cref{Triangle other grading} for the triangle case) try to obtain each inner ideal as the end of a $\mathbb Z$-grading.
	Of course, the inner ideal $\SS_+$ is the end of a $\mathbb Z$-grading on $\LL$.
	
	Let us now consider a $2$-dimensional inner ideal; then this inner ideal will also arise as the end of a $\mathbb Z$-grading on $\LL$.
	Indeed, note that $\ad_{T_{s_0}}$ is a grading derivation with components
	\begin{align*}
		\LL_{-3}&= \begin{psmallmatrix} 0 & 0 \\ 0 & k \end{psmallmatrix}_-\oplus \begin{psmallmatrix} 0 & 0 \\ 0 & k \end{psmallmatrix}_+\\[1ex]
		\LL_{-2}&= T_{\begin{psmallmatrix}
			0 & 0 \\ J & 0
		\end{psmallmatrix} }\\[1ex]
		\LL_{-1}&= \begin{psmallmatrix} 0 & J \\ 0 & 0 \end{psmallmatrix}_-\oplus \begin{psmallmatrix} 0 & J \\ 0 & 0 \end{psmallmatrix}_+\\[1ex]
		\LL_{-0}&=  \langle T_{\begin{psmallmatrix}
			1 & 0 \\ 0 & 0
		\end{psmallmatrix}}\rangle \oplus \langle T_{\begin{psmallmatrix}
			0 & 0 \\ 0 & 1
		\end{psmallmatrix}}\rangle\oplus \Inder(\A)\\[1ex]
		\LL_{1}&= \begin{psmallmatrix} 0 & 0 \\ J & 0 \end{psmallmatrix}_-\oplus \begin{psmallmatrix} 0 & 0 \\ J & 0 \end{psmallmatrix}_+\\[1ex]
		\LL_{2}&= T_{\begin{psmallmatrix}
			0 & J \\ 0 & 0
		\end{psmallmatrix} }\\[1ex]
		\LL_{3}&= \begin{psmallmatrix} k & 0 \\ 0 & 0 \end{psmallmatrix}_-\oplus \begin{psmallmatrix} k & 0 \\ 0 & 0 \end{psmallmatrix}_+, 	
	\end{align*}
	using $V_{\begin{psmallmatrix}
			0 & J \\ 0 & 0
		\end{psmallmatrix},\begin{psmallmatrix}
			0 & 0 \\ 0 & 1
		\end{psmallmatrix}}=V_{\begin{psmallmatrix}
			0 & 0 \\ J & 0
		\end{psmallmatrix},\begin{psmallmatrix}
			1 & 0 \\ 0 & 0
		\end{psmallmatrix}}=V_{\begin{psmallmatrix}
			1 & 0 \\ 0 & 0
		\end{psmallmatrix},\begin{psmallmatrix}
			1 & 0 \\ 0 & 0
		\end{psmallmatrix}}=V_{\begin{psmallmatrix}
			0 & 0 \\ 0 & 1
		\end{psmallmatrix},\begin{psmallmatrix}
			0 & 0 \\ 0 & 1
		\end{psmallmatrix}}=0$ and \cref{S1 Decom Inst}.
		So in this case, as opposed to the triangle case, we do not get a Jordan pair, and not even a so-called Kantor pair.
\end{remark}

\begin{remark}
    We have chosen to use the matrix structurable algebra $M(J,1)$, but we could also have chosen $M(J,\eta)$ for any parameter $\eta \in k^\times$.
    Indeed, the structurable algebras $M(J,1)$ and $M(J,\eta)$ are isotopic (see, e.g., \cite[Proposition~4.11 and Lemma~4.13]{Garibaldi2001}, where the result is stated for Albert algebras $J$ but holds in general).
    Since isotopic structurable algebras give rise to graded-isomorphic Lie algebras under the TKK-construction by \cite[Proposition~12.3]{Allison1981}, this does not affect the resulting geometry.
\end{remark}

Now we show that the generalized hexagon is, in fact, a Moufang hexagon and we determine the root groups in terms of the structurable algebra.
In the rest of this section, we fix a cycle $(x_0,x_1,\dots,x_{11},x_0)$ of length $12$ in $\Omega$ corresponding to the following cycle of length $6$ in the generalized hexagon:

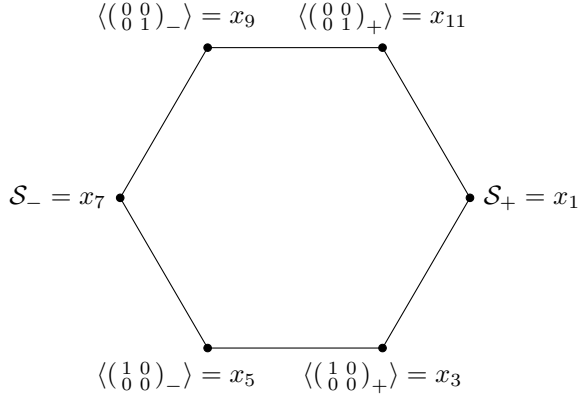
\begin{tikzpicture}
   \newdimen\R
   \R=2.3cm
   \draw (0:\R) \foreach \x in {60,120,...,360} {  -- (\x:\R) };
   \foreach \x/\l/\p in
     { 60/{$\langle \begin{psmallmatrix} 0&0\\0&1 \end{psmallmatrix}_+\rangle=x_{11}$}/above,
      120/{$\langle \begin{psmallmatrix} 0&0\\0&1 \end{psmallmatrix}_-\rangle=x_9$ \qquad \ }/above,
      180/{$\mathcal S_-=x_7$}/left,
      240/{$\langle \begin{psmallmatrix} 1&0\\0&0 \end{psmallmatrix}_-\rangle =x_5$ \qquad \ }/below,
      300/{$\langle \begin{psmallmatrix} 1&0\\0&0 \end{psmallmatrix}_+\rangle =x_3$}/below,
      360/{$\mathcal S_+=x_1$}/right
     }
     \node[inner sep=1pt,circle,draw,fill,label={\p:\l}] at (\x:\R) {};
	\end{tikzpicture}

So explicitly, we have $x_0=\langle \begin{psmallmatrix} 0&0\\0&1 \end{psmallmatrix}_+\rangle\oplus \mathcal S_+$, $x_1=\mathcal S_+$, etc.

Let $U_1,\dots, U_6$ be the root groups, as defined in \cref{Notation root groups}.
So $U_1$ is the subgroup of $\Aut (\Omega)$ which fixes all neighbors of $x_2$, $x_3$, $x_4$, $x_5$ and $x_6$, and similarly for the other root groups.
It turns out that, except for $U_1$, all these root groups are subgroups of $E_-(\A)$.
In order to determine the root groups explicitly, we can use \cref{Moufang general remark} to see that it suffices to show that the claimed subgroup (of $E_-(\A)$) fixes all the neighbors of a set of $5$ distinguished vertices and acts transitively on the set of neighbors of another distinguished vertex, with one neighbor excluded. 
We first determine $U_2$, $U_4$ and $U_6$ using the following lemma.
Recall the notations from \cref{E_i}.

\begin{lemma}
\label{1-dim root groups}
	Let $L$ be a non-degenerate Lie algebra generated by its extremal elements.
	Consider extremal elements $x$, $y$ such that $(x,y)\in E_{-1}$.
	Then $\exp (x)$ fixes all lines (in the extremal geometry) through $\langle y\rangle $. 
	In particular, $\exp (x)$ fixes~$\langle y\rangle $.
\end{lemma}
\begin{proof}
	First note that $\exp(x)$ is an automorphism, by \cite[Lemma 15]{Cohen2006}.
	Since $(x,y)\in E_{-1}$, we get $[x,y]=0$ and hence $\exp(x)$ fixes $y$.
	Consider $z$ extremal such that $(y,z)\in E_{-1}$.
	Then $(x,z)\in E_{\leq 0}$ or $(x,z)\in E_1$, by Theorem 28 of \textit{loc.\@ cit.\@} and the definition of a root filtration space.
	In the former case $[x,z]=0$, so $\exp(x)$ fixes the line $\langle y,z\rangle$.
	In the latter case $[x,z]\in \langle y\rangle$, by Lemma 1(ii) of \textit{loc.\@ cit.\@}, and thus $[x,[x,z]]=0$ and $\exp (x)$ fixes the line $\langle y,z\rangle$.
\end{proof}
\begin{corollary}
\label{Hexagon U246}
	For the root groups $U_2$, $U_4$ and $U_6$, we have
	\[ U_2=\{e_-(\begin{psmallmatrix} \lambda&0\\0&0 \end{psmallmatrix},0)\mid \lambda \in k\}, \ U_4=e_-(0,\SS), \ U_6=\{e_-(\begin{psmallmatrix} 0&0\\0&\lambda \end{psmallmatrix},0)\mid \lambda \in k\}.\]
	Moreover, $U_i$ acts transitively on the set of all neighbors of $x_i$ distinct from $x_{i+1}$, for $i=2,4,6$.
\end{corollary}
\begin{proof}
	As mentioned before, $\mathcal L=K(M(J,\eta))$ satisfies the conditions of \cref{1-dim root groups} by \cref{A nondeg} and \cref{Extremal J_+ explicit}.
	By \cref{1-dim root groups}, $e_-( \langle \begin{psmallmatrix} 1&0\\0&0 \end{psmallmatrix}\rangle,0)=\exp(x_5)$ fixes all neighbors of $x_3$, $x_5$ and $x_7$.
	The neighbors of $x_4$ are all the $1$-dimensional inner ideals contained in this $2$-dimensional inner ideal. 
	Since $e_-( \begin{psmallmatrix} 1&0\\0&0 \end{psmallmatrix},0)$ fixes $\begin{psmallmatrix} 1&0\\0&0 \end{psmallmatrix}_-$ and $\begin{psmallmatrix} 1&0\\0&0 \end{psmallmatrix}_+$ it fixes all neighbors of $x_4$. 
	A similar argument holds for $x_6$.
	Now we still need to check that the subgroup $U_2=\{e_-(\begin{psmallmatrix} \lambda&0\\0&0 \end{psmallmatrix},0)\mid \lambda \in k\}$ acts transitively on the set of all neighbors of $x_2=\mathcal S_+\oplus \langle \begin{psmallmatrix} 1&0\\0&0 \end{psmallmatrix}_+\rangle $ distinct from $x_3=\langle \begin{psmallmatrix} 1&0\\0&0 \end{psmallmatrix}_+\rangle $.
	Since $e_-(\begin{psmallmatrix} \lambda&0\\0&0 \end{psmallmatrix},0)(s_+)=s_+-\begin{psmallmatrix} \lambda&0\\0&0 \end{psmallmatrix}_+
	$, this is obvious.
	This shows the claim for $U_2$.
	The proof for $U_4$ and $U_6$ is similar.
\end{proof}

In order to determine the other root groups it is necessary to describe all lines through $x_5=\langle \begin{psmallmatrix} 1&0\\0&0 \end{psmallmatrix}_- \rangle$.
One can easily modify the argument to get a suitable description for all lines through $x_3, x_9$ and $x_{11}$ as well.

\begin{lemma}
\label{Lines through x8}
	The only proper inner ideals containing $x_5$ different from $x_6$ are 
	\[ x_5\oplus e_-(\begin{psmallmatrix} 0&x\\0&0 \end{psmallmatrix},0)(\begin{psmallmatrix} 1&0\\0&0 \end{psmallmatrix}_+), \text{ with } x\in J.\]
\end{lemma}
\begin{proof}
	First note that $e_-(\begin{psmallmatrix} 0&x\\0&0 \end{psmallmatrix},0)(\begin{psmallmatrix} 1&0\\0&0 \end{psmallmatrix}_+)$ has $(-1)$-component $\begin{psmallmatrix} 0&0\\-x^\sharp &0 \end{psmallmatrix}$ and generates a $1$-dimensional inner ideal collinear with $e_-(\begin{psmallmatrix} 0&x\\0&0 \end{psmallmatrix},0)(x_5)=x_5$.
	Now we show that there is a unique $1$-dimensional inner ideal collinear with $x_5$ and with $(-1)$-component $\begin{psmallmatrix} 0&0\\x&0 \end{psmallmatrix}$.
	Set $\varphi=e_+(\begin{psmallmatrix} 1&0\\0&0 \end{psmallmatrix},0)e_-(\begin{psmallmatrix} 0&0\\0&-1 \end{psmallmatrix},0)$.\footnote{Recall the convention from \cref{convention multiplication}.} 
	It is straightforward to compute that $\varphi$ maps $x_7=\SS_-$ onto $x_5$ and maps the extremal element $\begin{psmallmatrix} N(x)&x\\x^\sharp&1 \end{psmallmatrix}_-$ onto an extremal element with $(-1)$-component $\begin{psmallmatrix} N(x)&0\\x^\sharp&0 \end{psmallmatrix}$.
	Now the fact that $J$ is division together with \cref{Extremal J_+ explicit,Hexagon inner containing S_+} concludes this proof.
\end{proof}

\begin{lemma}
\label{Hexagon U35}
	For the root groups $U_3$ and $U_5$, we have 
	\[ U_3=\{e_-(\begin{psmallmatrix} 0&0\\x&0 \end{psmallmatrix},0)\mid x\in J\}, \ U_5=\{e_-(\begin{psmallmatrix} 0&x\\0&0 \end{psmallmatrix},0)\mid x\in J\}.\]
	Moreover, $U_i$ acts transitively on all neighbors of $x_i$ distinct from $x_{i+1}$, for $i=3, 5$.
\end{lemma}
\begin{proof}
	We show the claim for $U_5$.
	Let $x\in J$ be arbitrary.
	Set $\varphi=e_-(\begin{psmallmatrix} 0&x\\0&0 \end{psmallmatrix},0)$.
	Observe that $\varphi$ fixes all neighbors of $x_6$, $x_8$ and $x_{10}$.
	By \cref{Extremal J_+ explicit,Hexagon inner containing S_+} we see that any neighbor of $x_7=\SS_-$ is of the form $\SS_-\oplus \langle a_-\rangle$, for some $a\in \A$.
	Clearly $\varphi (\SS_-\oplus \langle a_-\rangle)=\SS_-\oplus \langle a_-\rangle$.
	Hence $\varphi$ also fixes all neighbors of $x_7$, i.e., all lines through $\mathcal S_-$. 
	We now need to check that $\varphi$ fixes all lines through $x_9=\langle \begin{psmallmatrix} 0&0\\0&1 \end{psmallmatrix}_-\rangle $.
	Consider $y=e_-(\begin{psmallmatrix} 0&0\\x'&0 \end{psmallmatrix},0)(\begin{psmallmatrix} 0&0\\0&1 \end{psmallmatrix}_+)$, with $x'\in J$ arbitrary.
	Then $\varphi (y)=y-T(x,x')x_9$.
	So $\varphi$ fixes all neighbors of $x_6$, $x_7$, $x_8$, $x_9$ and $x_{10}$.
	By \cref{Lines through x8}, we see that $\{e_-(\begin{psmallmatrix} 0&x\\0&0 \end{psmallmatrix},0)\mid x\in J\}$ acts transitively on the set of all neighbors of $x_5$ different from $x_6$.
\end{proof}

We now get to the final root group, which is the hardest to determine.
\begin{lemma}
\label{Hexagon U1}
	The root group $U_1$ is equal to
	\[ U_1 = \left\{ \exp \left( \ad \bigl( T_{\begin{psmallmatrix} 0&x\\0&0 \end{psmallmatrix}} \bigr) \right) \mid x\in J \right\} \]
	and acts transitively on the set of all neighbors of $x_1$ distinct from $x_2$.
\end{lemma}
\begin{proof}
	First we explain why $\varphi := \exp(\ad(T_{\begin{psmallmatrix} 0&x\\0&0 \end{psmallmatrix}}))$, with $x\in J$, is an automorphism.
	By \cite[Theorem 2.13, Theorem 3.4]{Stavrova2017} any endomorphism $\exp(\ad(l))$ is an automorphism if $l\in K(\A)$ is in the $(-1)$-component of a $5$-grading on $\LL=K(\A)$.
	Now we construct a $5$-grading on $\LL$ such that $T_{\begin{psmallmatrix} 0&x\\0&0 \end{psmallmatrix}}=L_{\begin{psmallmatrix} 0&x\\0&0 \end{psmallmatrix}}$ is in the $(-1)$-component.
	By \cite[Proposition 22]{Cohen2006}, there exists a $5$-grading on $\LL$ with grading derivation $[x_{11},x_5]=V_{\begin{psmallmatrix} 0&0\\0&1 \end{psmallmatrix},\begin{psmallmatrix} 1&0\\0&0 \end{psmallmatrix}}=:V$.
	Note that that $V \begin{psmallmatrix} 0&x\\0&0 \end{psmallmatrix}=\begin{psmallmatrix} 0&x\\0&0 \end{psmallmatrix}$ and $V^\epsilon (1)=\begin{psmallmatrix} -2&0\\0&1 \end{psmallmatrix}$.
	Then 
	\[ [V,T_{\begin{psmallmatrix} 0&x\\0&0 \end{psmallmatrix}}]=V_{V\begin{psmallmatrix} 0&x\\0&0 \end{psmallmatrix},1}+V_{\begin{psmallmatrix} 0&x\\0&0 \end{psmallmatrix},V^\epsilon(1)}=-T_{\begin{psmallmatrix} 0&x\\0&0 \end{psmallmatrix}},\] 
	using $V_{\begin{psmallmatrix} 0&x\\0&0 \end{psmallmatrix},\begin{psmallmatrix} 0&0\\0&1 \end{psmallmatrix}}=0$.
	Moreover, $[T_{\begin{psmallmatrix} 0&x\\0&0 \end{psmallmatrix}},T_{\begin{psmallmatrix} 0&x'\\0&0 \end{psmallmatrix}}]=0$ shows that we are indeed considering a subgroup of $\Aut(\LL)$, using \cite[Lemma 3.1.3]{Boelaert2019}.
	
	One easily checks:
	\begin{align}
		\label{T op 1} \varphi\begin{psmallmatrix} 1&0\\0&0 \end{psmallmatrix}_+&=\begin{psmallmatrix} 1&0\\0&0 \end{psmallmatrix}_+\\ 
		 \label{T op 2} T_{\begin{psmallmatrix} 0&x\\0&0 \end{psmallmatrix}} \begin{psmallmatrix} 0&0\\x'&0 \end{psmallmatrix}&=\begin{psmallmatrix} T(x',x)&0\\0&0 \end{psmallmatrix}\\ 
		 \label{T op 3}\varphi \begin{psmallmatrix} 0&0\\0&1 \end{psmallmatrix}_+&=\begin{psmallmatrix} N(x)&x\\x^\sharp &1 \end{psmallmatrix}_+,
	\end{align}
	for any $x'\in J$.
	By $[T_{\begin{psmallmatrix} 0&x\\0&0 \end{psmallmatrix}},s_+]=-\psi (\begin{psmallmatrix} 0&x\\0&0 \end{psmallmatrix},s)_+=0$, \cref{T op 1} and $T_{\begin{psmallmatrix} 0&x\\0&0 \end{psmallmatrix}}^\epsilon=-T_{\begin{psmallmatrix} 0&x\\0&0 \end{psmallmatrix}}$, $\varphi$ fixes all neighbors of $x_2$, $x_4$ and $x_6$.
	Consider $z:=e_-(\begin{psmallmatrix} 0&x'\\0&0 \end{psmallmatrix},0)(\begin{psmallmatrix} 1&0\\0&0 \end{psmallmatrix}_+)$.
	The $0$-component of $z$ is $-V_{\begin{psmallmatrix} 1&0\\0&0 \end{psmallmatrix},\begin{psmallmatrix} 0&x'\\0&0 \end{psmallmatrix}}$.
	We have 
	\[[T_{\begin{psmallmatrix} 0&x\\0&0 \end{psmallmatrix}},V_{\begin{psmallmatrix} 1&0\\0&0 \end{psmallmatrix},\begin{psmallmatrix} 0&x'\\0&0 \end{psmallmatrix}}]=V_{0,\begin{psmallmatrix} 0&x'\\0&0 \end{psmallmatrix}}-V_{\begin{psmallmatrix} 1&0\\0&0 \end{psmallmatrix},\begin{psmallmatrix} 0&0\\x\times x'&0 \end{psmallmatrix}}=0.\]
	The $(-1)$-component of $z$ equals $-\frac{1}{2}U_{\begin{psmallmatrix} 0&x'\\0&0 \end{psmallmatrix}}\begin{psmallmatrix} 1&0\\0&0 \end{psmallmatrix}=-\begin{psmallmatrix} 0&0\\x'^\sharp &0 \end{psmallmatrix}$.
	Hence, using \cref{T op 2}, we see that $\varphi(z)=z-T(x'^\sharp,x) x_5$ and so $\varphi$ fixes all neighbors of $x_5$ and similarly, all neighbors of $x_3$.
	Finally, \cref{T op 3} shows that $\{\exp(\ad(T_{\begin{psmallmatrix} 0&x\\0&0 \end{psmallmatrix}}))\mid x\in J\}$ acts transitively on the set of all neighbors of $x_1=\mathcal S_+$ different from $x_2$.
	Hence, \cref{Moufang general remark} concludes this proof.
	\end{proof}

\begin{lemma}
\label{Hexagon trans hex}
	Every cycle $(y_0,\dots,y_{11},y_0)$ of length $12$ in $\Omega$ can be mapped onto the cycle $(x_0,\dots,x_{11},x_0)$ by an element of $E(\A)$.
\end{lemma}
\begin{proof}
	We may assume that $y_1$ is $1$-dimensional.
	By \cref{Reducing to S_+}, the inner ideal $y_1$ can be mapped onto $\SS_+$.
	So we may assume $y_1=\SS_+=x_1$.
	Now $y_7$ is at distance~$3$ from $y_1$ in the extremal geometry.
	This is equivalent with $[\SS_+,[\SS_+,y_7]]\neq 0$, i.e. $y_7$ has non-zero $(-2)$-component.
	By \cref{-2 -1 0 1 2}, there exists an element of $E_+(\A)$ mapping $y_7$ onto $\SS_-$ and fixing $\SS_+$.
	So we may assume $y_7=\SS_-=x_7$.
	Since $y_5$ is collinear with $y_7$ and at distance $2$ from $y_1$, so $[\SS_+,[\SS_+,y_5]]=0$, we get by \cref{Extremal J_+ explicit,Hexagon U1} that we may assume $y_5=x_5$.
	Similarly, we may assume that $y_9=x_9$.
	Since $y_3$ is the unique neighbor of $y_1=x_1$ and $y_5=x_5$, we get $y_3=x_3$, and similarly $y_{11}=x_{11}$. 
\end{proof}

\begin{proposition}
	The generalized hexagon $\Omega$ is Moufang.
\end{proposition}
\begin{proof}
	This follows from \cref{Hexagon trans hex,Hexagon U246,Hexagon U35,Hexagon U1}. 
\end{proof}

We now determine the commutator relations.
\begin{lemma}
\label{commutator rel}
	The commutator relations between $U_2$, $U_3$, $U_4$, $U_5$ and $U_6$ are as in $E_-(\A)$.
	The commutator relations with $U_1$ are trivial except for:
	\begin{align*}
		&[\exp(\ad(T_{\begin{psmallmatrix} 0&x\\0&0 \end{psmallmatrix}})), e_-(\begin{psmallmatrix} 0&0\\0&\lambda \end{psmallmatrix},0)] \\
		      &\hspace*{8ex} = e_-\Bigl(\begin{psmallmatrix} \lambda N(x)&0\\0 &0 \end{psmallmatrix},0\Bigr) \; e_-\Bigl(\begin{psmallmatrix} 0&0\\-\lambda x^\sharp &0 \end{psmallmatrix},0 \Bigr) \; e_- \Bigl(0,-\lambda^2N(x)s\Bigr) \; e_-\Bigl(\begin{psmallmatrix} 0&\lambda x\\0 &0 \end{psmallmatrix},0\Bigr), \\
		&[\exp(\ad(T_{\begin{psmallmatrix} 0&x\\0&0 \end{psmallmatrix}})),e_-(\begin{psmallmatrix} 0&y\\0&0 \end{psmallmatrix},0)] \\
		      &\hspace*{8ex} = e_- \Bigl(\begin{psmallmatrix} -T(x^\sharp,y)&0\\0&0 \end{psmallmatrix},0\Bigr) \; e_-\Bigl(\begin{psmallmatrix} 0&0\\x\times y&0 \end{psmallmatrix},0\Bigr) \; e_-\Bigl(0,-T(x,y^\sharp)s\Bigr),\\
		&[\exp(\ad(T_{\begin{psmallmatrix} 0&x\\0&0 \end{psmallmatrix}})),e_-(\begin{psmallmatrix} 0&0\\y&0 \end{psmallmatrix},0)] = e_- \Bigl(\begin{psmallmatrix} T(x,y)&0\\0&0 \end{psmallmatrix},0\Bigr),
	\end{align*}
	for all $x,y\in J$ and $\lambda \in k$.
\end{lemma}
\begin{proof}
	We deduce the first commutator relation, the other two are obtained in a similar fashion.
	Consider $x\in J$ and $\lambda\in k$ arbitrary.
	Set 
	\begin{align*} 
		\varphi &=[\exp(\ad(T_{\begin{psmallmatrix} 0&x\\0&0 \end{psmallmatrix}})),e_-(\begin{psmallmatrix} 0&0\\0&\lambda \end{psmallmatrix},0)]\\
			&=\exp \bigl( \ad(T_{\begin{psmallmatrix} 0&-x\\0&0 \end{psmallmatrix}}) \bigr) \; e_- \bigl( \begin{psmallmatrix} 0&0\\0&-\lambda \end{psmallmatrix},0 \bigr) \; \exp \bigl( \ad(T_{\begin{psmallmatrix} 0&x\\0&0 \end{psmallmatrix}}) \bigr) \; e_- \bigl( \begin{psmallmatrix} 0&0\\0&\lambda \end{psmallmatrix},0 \bigr) .
	\end{align*}
	By \cite[Proposition 5.5]{Tits2002} there exist $a,b\in J$ and $\gamma, \mu\in k$ such that 
	\[ \varphi = e_- \bigl( \begin{psmallmatrix} \gamma &0\\0&0 \end{psmallmatrix},0 \bigr) \; e_- \bigl( \begin{psmallmatrix} 0&0\\a&0 \end{psmallmatrix},0 \bigr) \; e_- \bigl( 0,\mu s \bigr) \; e_- \bigl( \begin{psmallmatrix} 0 &b\\0&0 \end{psmallmatrix},0 \bigr)
	   = e_- \bigl( \begin{psmallmatrix} \gamma &b\\a&0 \end{psmallmatrix},(-\tfrac{1}{2} T(a,b)+\mu) s \bigr) . \]
	We get $\varphi e_- \bigl( \begin{psmallmatrix} 0 &0\\0&-\lambda \end{psmallmatrix},0 \bigr) = e_- \Bigl( \begin{psmallmatrix} \gamma &b\\a&-\lambda \end{psmallmatrix},(-\frac{1}{2}T(a,b)-\frac{1}{2}\lambda\gamma+\mu)s \Bigr)$, using 
	\[ \psi \bigl( \begin{psmallmatrix} \gamma &b\\a&0 \end{psmallmatrix},\begin{psmallmatrix} 0 &0\\0&-\lambda \end{psmallmatrix} \bigr) = -\lambda\gamma s.\]
	From \cref{T op 3}, we see that $(\varphi e_-(\begin{psmallmatrix} 0 &0\\0&-\lambda \end{psmallmatrix},0))(\id)$ equals $-\lambda \begin{psmallmatrix} -N(x) &-x\\x^\sharp& 1 \end{psmallmatrix}_-+\id$.
	On the other hand, $e_- \Bigl( \begin{psmallmatrix} \gamma &b\\a&-\lambda \end{psmallmatrix},(-\frac{1}{2}T(a,b)-\frac{1}{2}\lambda\gamma+\mu)s \Bigr)(\id)$ equals
	\[ (-T(a,b)-\lambda\gamma+2\mu)s_-+\begin{psmallmatrix} \gamma &b\\a&-\lambda \end{psmallmatrix}_-+\id.\]
	Hence $a=-\lambda x^\sharp$, $b=\lambda x$, $\gamma=\lambda N(x)$ and $\mu=\frac{1}{2}\lambda^2N(x)-\frac{1}{2}\lambda^2 T(x,x^\sharp)=-\lambda^2N(x)$.
\end{proof}

\begin{theorem}\label{hexagon root groups}
	Let $\Omega$ be the generalized hexagon from \cref{Hexagon Main Theorem}. Then $\Omega$ is the Moufang hexagon associated to the cubic Jordan division algebra $J$.
\end{theorem}
\begin{proof}
	Consider the following parametrization:
	\begin{align*}
		x_1(a) &= \exp \bigl( \ad(T_{\begin{psmallmatrix} 0&a\\0&0 \end{psmallmatrix}}) \bigr), &
		x_4(t) &= e_- \bigl( 0,-ts \bigr), \\
		x_2(t) &= e_- \bigl(\begin{psmallmatrix} t &0\\0&0 \end{psmallmatrix},0 \bigr), &
		x_5(a) &= e_- \bigl(\begin{psmallmatrix} 0 &a\\0&0 \end{psmallmatrix},0 \bigr), \\
		x_3(a) &= e_- \bigl(\begin{psmallmatrix} 0 &0\\a&0 \end{psmallmatrix},0 \bigr), &
		x_6(t) &= e_- \bigl(\begin{psmallmatrix} 0 &0\\0&-t \end{psmallmatrix},0 \bigr),
	\end{align*}
	for all $a\in J$ and $t\in k$.
	Using \cref{commutator rel}, we see that the commutator relations are the same as in \cite[(16.8)]{Tits2002}. 
\end{proof}

\begin{remark}
	The explicit description of the root groups is related to Peirce subspaces. 
	More precisely, if one considers the idempotent $e=\begin{psmallmatrix} 1 &0\\0&0 \end{psmallmatrix}$, and sets 
	\[ \A_{ij}=\{x\in\A\mid ex=ix, xe=jx\},\]
	for $i,j=0,1$, we get $U_2=e_-(\A_{11},0)$, $U_3=e_-(\A_{01},0)$, $U_5=e_-(\A_{10},0)$ and finally $U_6=e_-(\A_{00},0)$, where $\A=M(J)$, with $J$ a cubic Jordan division algebra.
	
	A similar remark applies to the triangle case. 
	In this case, consider $\A=F\oplus F$ as in \cref{construction exchange algebra}, with $F$ an alternative division algebra.
	Then $e=(1,0)$ is an idempotent and \cref{triangle root groups} yields $U_1=e_-(\A_{11},0)$ and $U_3=e_-(\A_{00},0)$.
\end{remark}

\bibliographystyle{alpha}
\bibliography{TDM-JM}

\end{document}